\documentclass[10pt]{amsart}

\usepackage{geometry,graphicx,amssymb,amsmath,amsbsy,eucal,amsfonts,mathrsfs,amscd,bm}
\usepackage[all]{xy}
\usepackage{tikz}
\usetikzlibrary{arrows,shapes,calc,snakes}
\usetikzlibrary{shapes,snakes}
\usepackage{tikz-3dplot}
\usetikzlibrary{intersections}
\usepackage{tkz-euclide}
\usetkzobj{all}
\usetikzlibrary{angles}
\usepackage{pgfplots}
\usepackage{mathtools}

\numberwithin{equation}{section}

\allowdisplaybreaks[4]
\theoremstyle{definition}

\newtheorem{theorem}{Theorem}[section]
\newtheorem{lemma}[theorem]{Lemma}
\newtheorem{corollary}[theorem]{Corollary}
\newtheorem{proposition}[theorem]{Proposition}

\theoremstyle{remark}
\newtheorem{remark}[theorem]{Remark}

\newcommand{\tr}{\text{tr}}
\newcommand{\we}{\wedge}
\newcommand{\R}{\mathbb{R}}

\newcommand{\la}{\lambda}

\newcommand{\p}{{\partial}}

\newcommand{\mct}{\mathcal{T}_h}

\newcommand{\curl}{{\ensuremath\mathop{\mathrm{curl}\,}}}
\newcommand{\grad}{{\ensuremath\mathop{\mathrm{grad}\,}}}
\newcommand{\dive}{{\ensuremath\mathop{\mathrm{div}\,}}}
\newcommand{\Div}{{\rm div}\,}

\newcommand{\pol}{\EuScript{P}}

\newcommand{\bld}[1]{\boldsymbol{#1}}

\newcommand{\bbR}{\mathbb{R}}

\newcommand{\La}{\Lambda}
\newcommand{\m}{\mu}
\newcommand{\s}{{ z}}

\newcommand{\range}{{\rm range}\,}
\newcommand{\bbV}{\mathbb{V}}
\newcommand{\bbE}{\mathbb{E}}
\newcommand{\bbF}{\mathbb{F}}
\newcommand{\bbT}{\mathbb{T}}

\newcommand{\Mrc}{M_{c,r-1}^1}
\newcommand{\Mrco}{\mathring{M}_{c,r-1}^1}
\newcommand{\Vrc}{V_{c,r-2}^2}
\newcommand{\Vrco}{\mathring{V}_{c,r-2}^2}

\title{Exact smooth piecewise polynomial sequences on Alfeld splits}

\author[G.~Fu, J.~Guzm\'an, and M.~Neilan]{Guosheng Fu, Johnny Guzm\'an and Michael Neilan}

\usepackage{hyperref}
\begin{document}

\maketitle

\begin{abstract}
We develop exact piecewise polynomial sequences
on Alfeld splits in any spatial dimension
and any polynomial degree.
An Alfeld split of a simplex is obtained by
connecting the vertices of an $n$-simplex
with its barycenter.  We show that, on these
triangulations, the kernel of the exterior derivative
has enhanced smoothness.  Byproducts of this theory
include characterizations of discrete divergence-free subspaces
for the Stokes problem, commutative projections, and simple
formulas for the dimensions of smooth polynomial spaces.
\end{abstract}

\section{Introduction}\label{section-1}

An Alfeld split (or refinement) of an $n$-dimensional simplex
is obtained by connecting each vertex of the simplex with
its barycenter \cite{Alfeld84,Schenck,Tanya13}.  This is also known as a barycenter
refinement in some communities \cite{OlshRebholz,Case11,Linke11}.
Such meshes are useful in several areas of computational mathematics.
For example, one can construct relatively low-order $C^1$ finite elements
on Alfeld splits.   This is the idea behind the famous (cubic) Clough-Tocher finite elements in two dimensions \cite{Clough},
and the (quintic) Alfeld elements in three dimensions \cite{Alfeld84}.
This family of triangulations have also been used to develop simple low-order, inf-sup stable,
and divergence-free yielding finite elements for the Stokes and Navier-Stokes problems; 
see \cite{ArnoldQin92} for two dimensional case and  \cite{Zhang04} for the three dimensional case. 


In this paper we will show that these $C^1$ finite elements
and Stokes finite element pairs
 are connected via an exact sequence consisting of piecewise polynomial spaces.
 The sequence is a de Rham complex, but where the finite element spaces have extra smoothness as compared to the canonical Whitney-N\'ed\'elec spaces; see \cite{AFW06,Nedelec80,Nedelec86}.

As a first step to prove these results, we take a single non-degenerate $n$-dimensional simplex $T$ ($n\ge 2$),
and consider the split (mesh) $T^\s$ which is obtained by adjoining the vertices of $T$ to its barycenter $\s$. 
We study $k$-forms with piecewise polynomial coefficients on these (local) meshes,
and show that the kernel of the exterior derivative has enhanced smoothness properties.
In particular, if $\omega$ is a piecewise polynomial $k$-form on $T^\s$ with vanishing
exterior derivative, then there exists a {\em continuous} piecewise polynomial 
$(k+1)$-form $\rho$ such that $\omega = d\rho$ (cf.~Theorems \ref{mainthm}--\ref{cor1}).
The case $k=n-1$ has been recently established in \cite{GuzmanNeilan18}.
This result allows us to develop $n$ new (local) de Rham complexes
consisting of piecewise polynomials.  In addition, a simple byproduct of this result
is a dimension formula of the space
of continuous, piecewise polynomial $k$-forms with continuous exterior derivative.
For example, we are able to recover the dimension of local $C^1$ spaces on Alfeld splits
by taking $k=0$ in this framework \cite{Tanya13}.  Another instance ($k=1$, $n=3$) is the dimension
of the space consisting of continuous piecewise polynomial vector-fields whose curl is continuous.


We then develop unisolvent sets of
degrees of freedom for the spaces in the complexes in three dimensions.
These degrees of freedom induce 
projections onto these spaces that commute with the differential operator,
and allow us to formulate the global finite element spaces and three global
discrete complexes with varying level of regularity.  One of the sequences  connects the $C^1$ finite element
space of Alfeld \cite{Alfeld84} to the inf-sup stable Stokes pair 
of Zhang \cite{Zhang04} globally.  This is done by introducing
a new $H^1({\rm curl})$-conforming finite element space that may
be useful for fourth order curl problems \cite{ZhengXu11}.
Finally, we show that the complexes are exact on contractible domains.


We  mention that Christiansen and Hu \cite{ChristiansenHu} have recently  
studied smoothed discrete de Rham sequences in any dimension. Their triangulations have different splits, and they only
considered low-order polynomial approximations in higher dimensions.

The paper is organized as follows: In Section \ref{section-2} we give preliminary results on differential forms on 
one simplex. In Section \ref{section-3} we  define finite element spaces on an Alfeld split of a single simplex. 
Important surjectivity properties of the exterior derivative are established. In Section \ref{sec-Smooth3D} we focus 
on the three dimensional case. We provide degrees of freedom of several finite element spaces that induce projections that satisfy commuting diagrams.  
In Section \ref{section-Global} we define the corresponding global finite element complexes. We show exactness properties on contractible domains.
Finally in Section \ref{section-Conclusion} we summarize our results and state possible future directions.

\section{Polynomial differential forms on a simplex}\label{section-2}
Let $T=[x_0, \ldots, x_n]$ be an $n$-simplex  with vertices $\{x_i\}_{i=0}^n$.
We denote by $\Delta_s(T)$ the set of $s$-simplices
of $T$.  We note that the cardinality of $\Delta_s(T)$ is $\binom{n+1}{s+1}$.
We let $t_i=x_i-x_0$ for $1 \le i \le n$ and assume that the determinant of the matrix $[t_1, \ldots, t_n]$ is positive. We let $\la_i$ for $0 \le i \le n$ be the barycentric coordinates for $T$,  that is, $\la_i$ is the  unique linear function such that $\la_i(x_j)=\delta_{ij}$,  $0 \le i, j \le n$. We  denote by $F_i$ the face of $T$ opposite to $x_i$, that is 
$F_i=[x_0, \ldots, \hat{x}_i, \ldots, x_n]$, where $\hat{\cdot}$ represents omission. Note that $\lambda_i$ vanishes on $F_i$. The differential $d \la_i: \mathbb{R}^n \rightarrow \mathbb{R}$ is given by 
$d \la_i(r)= \grad \la_i \cdot r$.  For integer $k\in [1,n]$,
and $0 \le \sigma(1)<\sigma(2)< \dots< \sigma(k) \le n$, we define 
the $k$-form $d\la_{\sigma(1)} \we d\la_{\sigma(2)} \we \cdots \we d \la_{\sigma(k)}$ as follows:
\begin{equation*}
(d\la_{\sigma(1)} \we d\la_{\sigma(2)} \we \cdots \we d \la_{\sigma(k)})(v_1, v_2, \ldots, v_k):= \det [d\la_{\sigma(i)}(v_j))],
\end{equation*}
where $v_1, \ldots, v_k \in \R^n$.

We define the space of polynomials on $T$ with respect to the barycentric coordinates:
\begin{equation*}
\pol_r(T):=\{ \sum_{|\alpha|\le r} a_{\alpha} \la_1^{\alpha_1} \cdots \la_n^{\alpha_n}: a_\alpha \in \R \},
\end{equation*}
and we use the convention $\pol_r(T) = \{0\}$ if $r$ is negative.
If $f=[x_{\tau(0)}, x_{\tau(1)}, \ldots, x_{\tau(s)}]\in \Delta_s(T)$ is a $s$ sub-simplex of $T$ where $\tau:\{0, 1, \ldots, s\} \rightarrow \{0,1, \ldots, n\}$ is increasing, then we define 
\begin{equation*}
\pol_r(f):=\{ \sum_{|\alpha|\le r} a_{\alpha} \la_{\tau(1)}^{\alpha_1} \cdots \la_{\tau(s)}^{\alpha_s}: a_\alpha \in \R \}.
\end{equation*}
Using the short-hand notation $d \la_{\sigma}=d\la_{\sigma(1)} \we d\la_{\sigma(2)} \we \cdots \we d \la_{\sigma(k)}$, 
we define the space of $k$-forms with polynomial coefficients on $T$ as follows:
\begin{equation*}
\pol_r \La^k(T):= \{ \sum_{\sigma\in \Sigma(k,n)}
a_{\sigma} d\la_{\sigma}:  a_{\sigma} \in \pol_r(T)\},
\end{equation*}
where $\Sigma(k,n)$ is the set of increasing maps $\{1,2,\ldots,n\}\to \{1,2,\ldots,k\}$.
If $f\in \Delta_s(T)$ with $s\ge k$, then
\begin{equation*}
\pol_r \La^k(f):= \{ \sum_{\sigma\in \Sigma(k,s)} a_{\sigma} d\la_{\tau \circ \sigma}:  a_{\sigma} \in \pol_r(f)\},
\end{equation*}
where here we used the notation $\tau \circ \sigma=\{\tau (\sigma(1)) , \tau(\sigma(2)), \ldots, \tau(\sigma(k))\}$.

For a polynomial $a \in \pol_r(T)$,  we see that the $1$-form $da$ is given as 
\begin{equation*}
da_{(\la_1, \ldots, \la_n)}= \sum_{j=1}^n \frac{\partial a}{\partial \la_j}(\la_1, \ldots, \la_n) d\la_j. 
\end{equation*}
If $\omega=\sum_{\sigma} a_{\sigma} d\la_{\sigma}  \in \pol_r \La^k(T)$ then
\begin{equation*}
d\omega=\sum_{\sigma} d a_{\sigma} \we d\la_{\sigma},
\end{equation*}
and therefore $d \omega \in  \pol_{r-1} \La^{k+1}(T)$.

The Koszul operator can be defined using barycentric coordinates:
\begin{equation*}
\kappa \omega=\sum_{\sigma} \sum_{i=1}^k (-1)^{i+1} a_{\sigma} \la_{\sigma(i)}  d\la_{\sigma(1)} \we \cdots \we \widehat{d \la_{\sigma(i)}} \we \cdots \we d \la_{\sigma(k)}.
\end{equation*}
Hence, $\kappa \omega \in \pol_{r+1} \La^{k-1}(T).$ 

Suppose that $\omega=\sum_{\sigma} a_{\sigma} d\la_{\sigma}  \in \pol_r \La^k(T)$ and suppose 
that $f=[\tau(0), \tau(1), \tau(2), \ldots, \tau(s)]$ is an $s$-simplex of $T$.  Then the trace 
of $\omega$ on $f$ is given by
%
\begin{equation*}
\tr_f \omega= \sum_{\sigma \subset \tau} \tr_f a_{\sigma} d\la_{\sigma}\in \pol_r \La^k(f),
\end{equation*}
where $\tr_f a_{\sigma}:=a_{\sigma}|_f$ is simply the restriction of $a_{\sigma}$ to $f$. We say that $\sigma \subset \tau$ if 
$\{\sigma(1), \ldots, \sigma(k)\} \subset \{\tau(0), \tau(1), \ldots, \tau(s)\}$.

We define the space
\begin{equation*}
\pol_{r}^- \La^k(f)= \pol_{r-1} \La^k(f)+ \kappa \pol_{r+1} \La^{k+1}(f). 
\end{equation*}

The following result is contained in \cite[Theorem 4.8]{AFW06}.
\begin{proposition}\label{Prop2}
Let  $\omega \in  \pol_r \La^k(T)$. Then if $r \ge 1$, $\omega$ is uniquely determined by
\begin{equation*}
\int_f \tr_f \omega  \we \eta \quad \text{ for all } \eta \in \pol_{r+k-s}^{-}\La^{s-k}(f),\ f \in \Delta_s(T),\ s \ge k.
\end{equation*}
\end{proposition}

We also need a result in the case $r=0$. To do this, we first state a result from Arnold et al. \cite[Lemma 4.6]{AFW06}.
\begin{proposition}\label{Prop1}
Let $\omega \in \pol_r \La^k(T)$. Suppose that $\tr_{F_i} \omega=0$, for $1 \le i \le n$ and 
\begin{equation*}
\int_T \omega \we \eta, \quad \text{ for all } \eta \in  \pol_{r-n+k} \La^{n-k}(T).
\end{equation*}
Then, $\omega =0$.  In particular, if $\omega\in \pol_0\La^k(T)$ with $k\le n-1$ satisfies
 $\tr_{F_i} \omega=0$ for $1 \le i \le n$, then $\omega=0$.
\end{proposition}

\begin{lemma}\label{lemma101}
Define the set of $k$-simplices that have $x_0$ as a vertex:
\[
S_k(T,x_0) :=\{f\in \Delta_k(T):\ x_0\in \Delta_0(f)\}.
\]
Then any $\omega\in \pol_0\Lambda^k(T)$ is uniquely determined by
\begin{equation}\label{eqn:101DOFs}
\int_f \tr_f \omega  \qquad  \text{for all } f \in S_k(T,x_0).
\end{equation}
\end{lemma}
\begin{proof}
We have that $\dim \pol_0 \La^k(T)= {n \choose k}$ which is exactly the 
cardinality of $S_k(T,x_0)$.  Thus to prove the result,
we show that if $\omega\in \pol_0\Lambda^k(T)$ vanishes on \eqref{eqn:101DOFs}, then $\omega =0$.
The result is clearly true if $k=n$ by Proposition \ref{Prop1}, and so we assume that $k\le n-1$.

Suppose that $\omega\in \pol_0\Lambda^k(T)$ vanishes on \eqref{eqn:101DOFs},
so that $\tr_f \omega=0$ for $f\in S_k(T,x_0)$. 
For any $f\in S_{k+1}(T,x_0)$ it is easy to see that the cardinalities of the sets $\Delta_k(f)$ and $\Delta_k(f)\cap S_k(x_0,T)$ are $(k+2)$
 and $(k+1)$, respectively.  Therefore, using Proposition \ref{Prop1}  we have
 $\tr_f \omega=0$ for all $f \in S_{k+1}(T, x_0)$. We  continue by induction to conclude that $\tr_f \omega=0$ for any  $f \in S_{n-1}(T, x_0)$.
 Finally we apply Proposition \ref{Prop1} once more to get $\omega=0$.
 \end{proof}

We will also need the following two lemmas.
\begin{lemma}\label{lemma1}
Suppose that $\omega\in \pol_r\Lambda^k(T)$ satisfies
$\tr_{F_i} \omega=0$ for some $i\in \{0,1,\ldots,n\}$.  Then,
\begin{equation*}
\omega= d \lambda_i \we v+ \lambda_i  w,
\end{equation*}
where $v \in \pol_{r} \La^{k-1}(T)$ and $w \in  \pol_{r-1} \La^{k}(T)$.
\end{lemma}
\begin{proof}
Without loss of generality we assume that $i=n$.

We note that we can write $\omega$ is the following form:
\begin{equation*}
\omega = \sum_{\sigma \in \Sigma(k,n)} a_\sigma d\la_{\sigma} = \sum_{\sigma\in \Sigma(k,n-1)} a_\sigma d\lambda_\sigma
+ \mathop{\sum_{\sigma\in \Sigma(k,n)}}_{\sigma(k)=n} a_\sigma d\lambda_{\sigma(1)} \wedge \cdots \wedge d\la_{\sigma(k-1)} \wedge  d\lambda_n
\end{equation*}
with $a_\sigma\in \pol_r(T)$.
We then have $0 = \tr_{F_n} \omega = \sum_{\sigma\in \Sigma(k,n-1)} (\tr_{F_n} a_\sigma)  d\lambda_\sigma$,
and so $\tr_{F_n} a_\sigma=0$ for all $\sigma\in \Sigma(k,n-1)$.  Therefore $a_\sigma = \lambda_n b_\sigma$
for some $b_\sigma\in \pol_{r-1}(T)$.  The desired result now follows upon setting
\[
v = (-1)^{k-1}\mathop{\sum_{\sigma\in \Sigma(k,n)}}_{\sigma(k)=n} a_\sigma d\lambda_{\sigma(1)} \wedge \cdots \wedge d\lambda_{\sigma(k-1)},\quad
w = \sum_{\sigma\in \Sigma(k,n-1)} b_\sigma d\lambda_\sigma.
\]
\end{proof}

\begin{lemma}\label{zerotrace}
Suppose that $\omega=d\la_i \we v$ with $v \in \pol_r \La^{k-1}(T)$ for some $i\in \{0,1,\ldots,n\}$.  Then if $\tr_{F_i} v=0$,
there holds $\omega|_{F_i}=
0$, and so  $\omega=\lambda_i w$ for some $w \in \pol_{r-1} \La^k(T)$. Conversely, if  $\omega|_{F_i}=
0$ then we have  $\tr_{F_i} v=0$. 
\end{lemma}
\begin{proof}
Without loss of generality, we assume that $i=n$.  Let
$t_j = x_j-x_0$ for $1\le j\le n$.
Then $\{t_j\}_{j=1}^n$ forms a basis of $\R^n$,
and $\{t_j\}_{j=1}^{n-1}$ is a basis of the tangent
space of $F_n$.  We also have $d\la_n (t_n)=1$
and $d\la_n(t_j)=0$ for $1\le j\le n-1$.

Suppose that $\omega = d\lambda_n\wedge v$ with $v\in \pol_r\Lambda^{k-1}(T)$ and $\tr_{F_n} v=0$.
For  vectors $r_i\in \R^n\ (1\le i\le k)$, we can write $r_i = \sum_{j=1}^{n} a_{ij} t_j$ where $a_{ij}\in \R$.
Then, if $x\in F_n$, we have
\begin{alignat*}{1}
\omega _x(r_1, r_2, \ldots, r_k)=&  \sum_{j_1=1}^{n} \cdots \sum_{j_k=1}^{n} a_{1 j_1} a_{2 j_2} \cdots a_{k j_k}  \omega_x(t_{j_1}, t_{j_2}   \ldots, t_{j_k}) \\
%
%
 =&   \mathop{\sum_{1 \le j_1, \ldots, j_k \le n}}_{j_m \neq j_s \text{ for } m \neq s}  a_{1 j_1} a_{2 j_2} \cdots a_{k j_k}   (d \la_n \we v)_x (t_{j_1}, t_{j_2}   \ldots, t_{j_k}) \\
 =&  \mathop{\sum_{1 \le j_1, \ldots, j_k \le n}}_{j_m \neq j_s \text{ for } m \neq s}  a_{1 j_1} a_{2 j_2} \cdots a_{k j_k}   \sum_{\ell=1}^k (-1)^{\ell-1} d \la_n (t_{j_\ell})   v_x (t_{j_1}, \ldots, \widehat{t_{j_\ell}},   \ldots, t_{j_k}) \\
  =&  \mathop{\sum_{1 \le j_1, \ldots, j_k \le n}}_{j_m \neq j_s \text{ for } m \neq s}  a_{1 j_1} a_{2 j_2} \cdots a_{k j_k}   \mathop{\sum_{\ell=1}^k}_{j_\ell=n} (-1)^{\ell-1}    \tr_{F_n} v_x (t_{j_1}, \ldots, \widehat{t_{j_\ell}},   \ldots, t_{j_k})  \\
  =& 0.
\end{alignat*}
 Since $x \in F_n$ and $r_1, \ldots, r_k \in \R^n $ where arbitrary, we conclude that $\omega|_{F_n} =0$. 

Now assume that $\omega |_{F_n}=0$. Then for any $x \in F_n$   we have 
 \begin{equation*}
 0= \omega_x(t_n, t_{j_1}, \ldots, t_{j_{k-1}})= \tr_{F_n} v_x (t_{j_1}, t_{j_2}, \ldots, t_{j_{k-1}}) 
 \end{equation*}
 for any $1 \le j_1< \ldots < j_{k-1} \le n-1$. This implies that $\tr_{F_n} v=0$.   
\end{proof}

\section{Polynomial differential forms on an Alfeld split}\label{section-3}
Here, we apply the results of the previous 
section to derive some exactness properties
of polynomial differential forms on an Alfeld split simplex.
As before, we let $T = [x_0,\ldots,x_n]$ be an $n$-simplex.
We set $z = \frac1{n+1}\sum_{i=0}^n x_n$ 
to be the barycenter of $T$,
and we subdivide $T$ into $(n+1)$ $n$-simplices
by adjoining the vertices of $T$ with  $z$.
Namely,  we set $T_i=[z, x_0, \cdots, \widehat{x_i}, \cdots, x_{n}]$ so that
$\bar T= \cup_{i=0}^{n} \bar T_i$. We set $T^\s=\{ T_0, \ldots, T_{n}\}$ to be the mesh of this sub-division. 
We denote the set of $s$-dimensional simplices in $T^\s$ as
$\Delta_s(T^\s) = \{f\in \Delta_s(T_i):\ T_i\in T^\s\}$.
The cardinality of this set is given by 
\[
\# \Delta_s(T^\s)=
\begin{cases}
\binom{n+2}{s+1} \quad \text{ for } s\le n-1 ,\\
n+1 \quad \text{ for } s=n.
\end{cases}
\]
We let $\m$ be the hat function associated with the barycenter $z$, that is, $\m$ is uniquely determined
by the conditions
$\m|_{T_i}\in \pol_1(T_i)$, $\m\in H^1_0(T)$, and $\m(z)=1$.  We denote by $\m_i$ the restriction
of $\m$ to $T_i$, and we note that $\m_i = (n+1)\lambda_i$ for $i=0,1,\ldots,n$.

We define the following local spaces:
\begin{alignat*}{1}
V_r^k(T^\s): =&\{ \omega \in L^2 \La^k(T):   \omega|_T \in \pol_r \La^k(T_i) \text{ for } 0 \le i \le n \}, \\
V_{d,r}^k(T^\s): =&\{ \omega \in V_r^k(T^\s): d \omega \in L^2 \La^{k+1} (T) \}, \\
M_r^k(T^\s) := &  \{ \omega \in C^0 \La^k(T):  \omega|_T \in \pol_r\La^k(T_i) \text{ for }  0 \le i \le n \}, \\
M_{d,r}^k(T^\s):=&  \{ \omega \in M_r^k(T^\s): d \omega \in C^0 \La^{k+1}(T) \}. 
\end{alignat*}
We also define the analogous spaces with homogenous boundary conditions for $k\le n-1$:
\begin{alignat*}{1}
\mathring{V}_{d,r}^k(T^\s) :=&\{ \omega \in V_{d,r}^k(T^\s): \tr_{F} \omega =0  \text{ for all } F \in \Delta_{n-1}(T)\}, \\
\mathring{M}_r^k(T^\s) := &  \{ \omega \in M_r^k(T^\s)  :  \omega|_{\partial T} =0\}, \\
\mathring{M}_{d,r}^k(T^\s):=&  \{ \omega \in M_{d,r}^k(T^\s): \omega|_{\partial T}=0, d\omega|_{\partial T}=0  \}.
\end{alignat*}
In the case $k=n$, we will also set average zero constraints:
\begin{alignat*}{1}
\mathring{V}_{d,r}^n(T^\s):=&\{\omega\in V_{d,r}^n(T^\s):\ \int_T \omega=0\}, \\
\mathring{M}_{d,r}^n(T^\s)=\mathring{M}_r^n(T^\s):= &\{\omega \in {M}_r^n(T^\s): \omega|_{\partial T}=0,  \int_T \omega=0\}.
\end{alignat*}

It is well known (see, for example, \cite{AFW06})
that $\tr_f \omega$ is single-valued for $\omega\in V^k_{d,r}(T^\s)$ for all 
$f\in \Delta_s(T^\s)$ and $s\ge k$.
Moreover, if $\omega \in V_{d,r}^k(T^\s)$ (resp., $\omega\in \mathring{V}_{d,r}^k(T^\s)$)
with $d \omega=0$, then there exists a $\rho  \in V_{d,r+1}^{k-1}(T^\s)$ (resp., $\rho\in  \mathring{V}_{d,r+1}^{k-1}(T^\s))$ such that $d\rho=\omega$. 
The goal of this section is to prove the following related results.

\begin{theorem}\label{mainthm}
Suppose that $\omega \in \mathring{V}_{d,r}^k(T^\s)$ and $d\omega=0$. 
Then there exists a $\rho \in \mathring{M}_{r+1}^{k-1}(T^\s)$ such that $d \rho=\omega$.
\end{theorem}
%
%
\begin{theorem}\label{cor1}
Suppose that $\omega \in V_{d,r}^k(T^\s)$ and $d\omega=0$.  Then there exists a $\rho \in M_{r+1}^{k-1}(T^\s)$ such that $d \rho=\omega$.
\end{theorem}

\begin{remark}
The case $k=n-1$ in Theorem \ref{mainthm},
which corresponds to local finite element
pairs for the Stokes problem,
has been established in \cite{ArnoldQin92,Zhang04,GuzmanNeilan18}. 
\end{remark}

The next corollaries easily follow.

\begin{corollary}\label{cor2}
Suppose that $\omega \in \mathring{M}_{r}^k(T^\s)$ and $d\omega=0$. Then there exists a $\rho \in \mathring{M}_{d,r+1}^{k-1}(T^\s)$ such that $d \rho=\omega$.
\end{corollary}


\begin{corollary}\label{cor3}
Suppose that $\omega \in \mathring{M}_{d,r}^k(T^\s)$ and $d\omega=0$.  Then there exists a $\rho \in \mathring{M}_{d,r+1}^{k-1}(T^\s)$ such that $d \rho=\omega$.
\end{corollary}

\begin{corollary}\label{cor4}
There holds
\begin{align}
\label{eqn:DimCount1}
\dim M_{d,r}^k(T^\s) & = \dim M^{k+1}_{r-1}(T^\s) + \dim M_r^k(T^\s) - \dim V_{d,r-1}^{k+1}(T^\s),\\
\label{eqn:DimCount2}
\dim \mathring{M}_{d,r}^k(T^\s) & = \dim \mathring{M}^{k+1}_{r-1}(T^\s) + \dim \mathring{M}_r^k(T^\s) - \dim \mathring{V}_{d,r-1}^{k+1}(T^\s).
\end{align}
\end{corollary}

\begin{remark}\label{rem:CountingFun} Using Corollary \ref{cor4}, one can obtain explicit formulas for the dimensions of $M_{d,r}^k(T^\s)$ and  $\mathring{M}_{d,r}^k(T^\s)$
in terms of $r$, $k$, and $n$.
To this end, we can easily show that 
\begin{equation*}
\dim M_r^0(T^\s)=\sum_{s=0}^{n} \# \Delta_s(T^\s) \dim \pol_{r-s-1}\Lambda^s(\bbR^s)=(n+1) \binom{r-1}{n} + \sum_{s=0}^{n-1} \binom{n+2}{s+1}\binom{r-1}{s},
\end{equation*}
where we used that 
\begin{equation*}
\dim \pol_{r-s-1}\Lambda^s(\bbR^s)= \binom{r-1}{s}.
\end{equation*}
Hence, since  $\dim M_r^k(T^\s) = \binom{n}{k} \dim M_r^0(T^\s)$, we have
\begin{align*}
\dim M_r^k(T^\s)
& = \binom{n}{k}\Big[(n+1) \binom{r-1}{n} + \sum_{s=0}^{n-1} \binom{n+2}{s+1}\binom{r-1}{s}\Big]\\
  & = \binom{n}{k} \Big[\binom{r+n+1}{n+1} - \binom{r}{n+1}\Big].
\end{align*}

Likewise, Proposition \ref{Prop2} implies that (see \cite{AFW06} for details)
\begin{align*}
\dim V_{d,r}^k(T^\s) 
&= \sum_{s=k}^n \# \Delta_s(T^\s) \dim \pol^-_{r+k-s}\Lambda^{s-k}(\bbR^s)\\
& = (n+1) \binom{r-1}{n-k} \binom{r+k}{k} +\sum_{s=k}^{n-1} \binom{n+2}{s+1} \binom{r-1}{s-k} \binom{r+k}{k}\\
& = \binom{r+k}{k}\Big[\binom{r+n+1}{n-k+1} - \binom{r}{n+1-k}
\Big].
\end{align*}
We then find
\begin{align*}
\dim M_{d,r}^k(T^\s)
& = \binom{n}{k+1}\Big[ \binom{r+n}{n+1} - \binom{r-1}{n+1}\Big]
+ \binom{n}{k}\Big[\binom{r+n+1}{n+1}-\binom{r}{n+1}\Big]\\
&\qquad - \binom{r+k}{k+1}\Big[ \binom{r+n}{n-k} -\binom{r-1}{n-k}\Big].
\end{align*}
Similar arguments also show that
%
\begin{align*}
\dim \mathring{M}_{d,r}^k(T^\s)
& = \binom{n}{k+1}\Big[ \binom{r+n-1}{n+1} - \binom{r-2}{n+1}\Big]
 +\binom{n}{k}\Big[ \binom{r+n}{n+1} - \binom{r-1}{n+1}\Big]\\
&\qquad - \binom{r+k}{k+1}\Big[\binom{r+n-1}{n-k}- \binom{r-2}{n-k}\Big].
 \end{align*}
\end{remark}

\begin{remark}
Corollary \ref{cor4} gives, for example,
the local dimension of $C^1$ elements
on an Alfeld split.  Namely, taking
$k=0$ in the dimension count yields
\begin{align*}
\dim M_{d,r}^0(T^\s) 
%
%
%
& = n \Big[\binom{r+n}{n+1} - \binom{r-1}{n+1}\Big] + \Big[ \binom{r+n+1}{n+1} - \binom{r}{n+1}\Big] - r\Big[ \binom{r+n}{n} - \binom{r-1}{n}\Big]\\
&  = \binom{r+n}{n}+ n \binom{r-1}{n}.
\end{align*}
This dimension count has also been established in \cite{Tanya13} (also see \cite{Schenck})
using different arguments.  Note that, since $\dim \pol_r(T) = \binom{r+n}{n}$,
 $M_{d,r}^0(T^\s) = \pol_r(T)$ for $r\le n$.
\end{remark}

%
\subsection{Preliminary results} Before proving the main results in this section, we need some preliminary results.
We start with a well-known result stating that the traces of forms in $V_{d,r}^{k}(T^\s)$ are single-valued.

\begin{proposition}\label{prop201}
If $\omega \in  V_{d,r}^{k}(T^\s)$, then $\tr_f \omega$ is single valued for any sub-simplex $f \in \Delta_s(T^\s)$ for $s \ge k$.  
In particular, let $T_1, T_2 \subset T^\s$, $\omega_i= \omega|_{K_i}$, and suppose that $f \subset \Delta_s(T_1), \Delta_s(T _2)$.
Then if $r_1, r_2 \ldots, r_k \in \mathbb{R}^n$ are tangent to $f$, then
\begin{equation*}
(\omega_1)_x(r_1, \ldots, r_k)= (\omega_2)_x(r_1, \ldots, r_k) \text{ for all } x \in f.
\end{equation*}
\end{proposition}

Next, we prove an analogue of Lemma \ref{lemma1}
on an Alfeld split.
\begin{lemma}\label{lemma2}
Any $\omega \in \mathring{V}_{d,r}^k(T^\s)$ satisfies 
\begin{equation}\label{rep}
\omega= d \m \we v+ \m  w
\end{equation}
for some $v \in V_r^{k-1}(T^\s)$ and $w \in  V_{r-1}^{k}(T^\s)$. Moreover, $\tr_f v$ is single-valued for all $f \in \Delta_s(T)$ 
and $s\ge k-1$.
\end{lemma}
\begin{proof}
Applying Lemma \ref{lemma1} to each $T_i \in T^\s$ and recalling that $\m_i = (n+1)\lambda_i$,
we get the representation  \eqref{rep}. Moreover, the value  $\tr_F v$ is clearly single-valued
for $F\in \Delta_{n-1}(T)$.

Let  $f \in \Delta_s(T)$ with $ k-1\le s\le n-2$.
Let $K_1, K_2 \in T^\s$ such that $f \in \Delta_s(K_1)$ and $f \in \Delta_s(K_2)$, and set $v_i =v|_{K_i}$.    
Writing $f=[x_{\tau(0)}, x_{\tau(1)} , \ldots , x_{\tau(s)}]$,  we define $f'=[z, x_{\tau(0)}, x_{\tau(1)} , \ldots , x_{\tau(s)}]$,
 and note that $f' \in  \Delta_{s+1}(K_1)$ and $f' \in \Delta_{s+1}(K_2)$.

Let $\{r_i\}_{i=1}^{k-1}\subset \R^n$ be linearly independent vectors
that are tangent to $f$, and set $t = z-x_{\tau(0)}$.  Fix an arbitrary point $x\in f$, and note that $\m(x)=0$ because
$f\subset \partial T$ and $\m\in H^1_0(T)$.  It then follows from the representation 
\eqref{rep} that $\omega_x = (d\m \wedge v)_x$.
We also note that $\{r_i\}_{i=1}^{k-1}$ and $t$ are tangent
to $f'$,
and it thus follows that the quantity $\omega_x(t,r_1,\cdots,r_{k-1})$ is single-valued
because  $\omega  \in \mathring{V}_{d,r}^k(T^\s)$.
Using these two properties and the identities $d\mu(t) = 1$ and $d\mu(r_i)=0$, we find that
\begin{alignat*}{1}
\tr_f(v_1)_x( r_1, \ldots, r_{k-1})=& (v_1)_x( r_1, \ldots, r_{k-1}) \\
=&(d \m \we v_1 )_x( t, r_1, \ldots, r_{k-1})\\
=&\omega_x(t, r_1, \ldots, r_{k-1}) \\
=&  (d \m \we v_2 )_x( t, r_1, \ldots, r_{k-1}) \\
 =&(v_2)_x( r_1, \ldots, r_{k-1}) \\
 =& \tr_f(v_2)_x( r_1, \ldots, r_{k-1}).
\end{alignat*}
Thus, $\tr_f v$ is single-valued.
\end{proof}

The following lemma will be crucial.
\begin{lemma}\label{crucial}
Let $\omega \in \mathring{V}_{d,r}^k(T^\s)$ and  let $\ell \ge 0$ be an integer. If $r \ge 1$,  then there exists  $\gamma \in \pol_{r}\La^{k-1}(T)$ and $\psi \in V_{d,r-1}^k(T^\s)$ such that 
\begin{equation*}
\m^{\ell} \omega=d(\m^{\ell+1} \gamma)+ \m^{\ell+1} \psi.
\end{equation*}
Let $r=0$ and in addition if $k=n$ assume  that $\int_T \m^{\ell} \omega=0$. Then there exists a  $\gamma \in \pol_{0}\La^{k-1}(T)$ such that 
\begin{equation*}
\m^{\ell} \omega=d(\m^{\ell+1} \gamma).
\end{equation*}
\end{lemma}
\begin{proof}
Let us first consider the case $r \ge 1$. By Lemma \ref{lemma2} we have 
\begin{equation}\label{rep1}
\omega= d \m \we v+ \m  w,
\end{equation}
where $v \in V_r^{k-1}(T^\s)$ and $w \in  V_{r-1}^{k}(T^\s)$. Moreover, $\tr_f v$ is single valued for all $f \in \Delta_s(T)$ 
with $s\ge k-1$.
According to Proposition \ref{Prop2}, there exists a unique $\gamma \in \pol_{r} \La^{k-1}(T)$ such that 
\begin{alignat*}{2}
(\ell+1) \int_f \tr_f \gamma  \we \eta&= \int_f \tr_f v  \we \eta \quad &&\text{ for all } \eta \in \pol_{r+k-1-s}^{-}\La^{s-k+1}(f),\ f \in \Delta_s(T),\ k-1\le s\le n-1,\\
\intertext{and}
\int_T  \gamma  \we \eta&=0  &&\text{ for all } \eta \in \pol_{r+k-1-n}^{-}\La^{n-k+1}(T).
\end{alignat*}
It then follows from Proposition \ref{Prop2} that  $(\ell+1)\tr_F  \gamma =\tr_F v$ for all $F \in \Delta_{n-1}(T)$.   
Hence, by Lemma \ref{zerotrace} we have that 
\begin{equation*}
d \m \we v =(\ell +1) d\m \we \gamma+ \m \phi
\end{equation*}
for some $\phi \in V_{r-1}^k(T^\s)$.
Using this identity and the Leibniz rule
\begin{equation*}
d(\m^{\ell+1} \gamma)=(\ell+1) \m^\ell d\m \we \gamma+ \m^{\ell+1} d \gamma, 
\end{equation*}
we have
\begin{align*}
\mu^\ell (d \mu \wedge v)
 = (\ell+1) \mu^\ell  d \mu\wedge \gamma + \mu^{\ell+1} \phi = d (\mu^{\ell+1}\gamma) - \mu^{\ell+1} d \gamma+ \mu^{\ell+1} \phi.
\end{align*}
It then follows from \eqref{rep1} that
\begin{equation*}
\m^\ell\omega=d(\m^{\ell+1} \gamma)+ \m^{\ell+1} \psi 
\end{equation*}
with $\psi=-d \gamma+\phi+w\in V_{r-1}^k(T^\s)$. 
{Finally, since $\tr_f (\m^\ell \omega)$ and
$\tr_f \big(d (\mu^{\ell+1}\gamma))$ are single-valued 
on $f\in \Delta_s(T^\s)$ for $s\ge k$,
we conclude that $\tr_f(\mu^{\ell+1} \psi)$ is single-valued.
Therefore $\psi\in V_{d,r-1}^k(T^\s)$.}
This proves the result in the case $r \ge 1$.  

For the case $r=0$, we  have \eqref{rep1} with $w =0$.  Applying Lemma \ref{lemma101},
we uniquely determine 
$\gamma \in \pol_{0} \La^{k-1}(T)$ by the conditions
\begin{equation*}
(\ell+1) \int_f \tr_f \gamma = \int_f \tr_f v  \qquad \text{ for all } f \in S_{k-1}(T,x_0).
\end{equation*}
In this way, applying Proposition \ref{Prop2}
we have  $(\ell+1)\tr_F  \gamma =\tr_F v$ for all $F \in S_{n-1}(T,x_0)$.   Hence, by Lemma \ref{zerotrace} we have that 
\begin{equation*}
\xi=0 \qquad  \text{ on } T_i, 1 \le i \le n,
\end{equation*}
where $\xi:=\omega -(\ell +1) d \m \we \gamma \in V_{d, 0}^{k}(T^z)$. Hence, $\tr_F \xi=0$ for all $F \in \Delta_{n-1}(T_0)$. If $k \le n-1$, we apply
Proposition \ref{Prop2} to get $\xi=0$ on $T_0$, and therefore $\xi=0$ on $T$.
If $k=n$, we use the assumption that $\int_T \m^\ell \omega=0$ to obtain 
\begin{equation*}
0=\int_T (\m^\ell \omega-d(\m^{\ell+1} \gamma))= \int_T \m^\ell \xi=\int_{T_0} \m^\ell \xi.
\end{equation*}
This implies that $\xi=0$ on $T_0$ and hence $\xi=0$ on all of $T$. 
Finally, we finish the proof by
applying the product rule:
\begin{equation*}
d(\m^{ \ell+1} \gamma) =(\ell+1) \m^\ell d\mu \we \gamma=\m^\ell \omega.
\end{equation*}
\end{proof}



\subsection{Proof of Theorem \ref{mainthm}}

Let $\omega \in \mathring{V}_{d,r}^k(T^\s)$ and $d\omega=0$. Assume we have found $ \gamma_{r}, \ldots, \gamma_{r-j}$
with $\gamma_{\ell} \in \pol_{\ell} \La^{k-1}(T)$  
and $\omega_{r-(j+1)} \in  V_{d,r-(j+1)}^k(T^\s)$ such that 
\begin{equation*}
\omega=d(\m \gamma_r+ \m^2 \gamma_{r-1}+ \cdots+\m^{j+1} \gamma_{r-j}) + \m^{j+1} \omega_{r-(j+1)}.
\end{equation*}
 Then, we see that 
 \begin{equation*}
 0=d(\m^{j+1} \omega_{r-(j+1)})= \m^j(\m d\omega_{r-(j+1)} + (j+1) d \m \we \omega_{r-(j+1)}),
 \end{equation*}
 which implies that $\m d\omega_{r-(j+1)} + (j+1) d \m \we \omega_{r-(j+1)}=0$ on $T$. Hence, we have that $d \m \we \omega_{r-(j+1)}=0$ on $\partial T$. Using Lemma \ref{zerotrace} we have  $\tr_{F} \omega_{r-(j+1)}=0$ for all $F \in \Delta_{n-1}(T)$. Or in other words, $\omega_{r-(j+1)} \in  \mathring{V}_{d,r-(j+1)}^k(T^\s)$.
We then apply Lemma \ref{crucial} to get
\begin{equation*}
\m^{j+1} \omega_{r-(j+1)}=d(\m^{j+2} \gamma_{r-(j+1)})+ \m^{j+2} \omega_{r-(j+2)},
\end{equation*} 
where  $\gamma_{r-(j+1)} \in \pol_{r-(j+1)}\La^{k-1}(T)$ and $\omega_{r-(j+2)} \in V_{d,r-(j+2)}^k(T^\s)$. 
It follows that
\begin{equation*}
\omega=d(\m \gamma_r+ \m^2 \gamma_{r-1}+ \cdots+\m^{j+1} \gamma_{r-j}+ \m^{j+2} \gamma_{r-(j+1)}) + \m^{j+2} \omega_{r-(j+2)}.
\end{equation*}
Continuing by induction we have
\begin{equation*}
\omega=d(\m \gamma_r+ \m^2 \gamma_{r-1}+ \cdots+\m^{r} \gamma_{1})+ \m^r \omega_0.
\end{equation*}
Note that if $k=n$, then we have $\int_{T} \omega=0$ and so $\int_T \m^r \omega_0=0$. We can apply Lemma \ref{crucial} to write 
\begin{equation*}
\m^r \omega_0= d(\m^{r+1} \gamma_0),
\end{equation*}
for some $\gamma_0 \in \pol_0 \La^k(T)$. This completes the proof. \hfill $\qed$

\subsection{Proof of Theorem \ref{cor1}}
Let $\omega \in V_{d,r}^k(T^\s)$ with $d \omega=0$. We will consider the case $r \ge 1$ first.  Define $\Pi \omega \in \pol_r \La^k(T)$ such that 
\begin{equation*}
\int_f \tr_f \Pi \omega  \we \eta_f =  \int_f \tr_f \omega  \we \eta_f \quad \text{ for all } \eta_f \in \pol_{r+k-s}^{-}\La^{s-k}(f),\ f \in \Delta_s(T),\ s \ge k.
\end{equation*}
This is the canonical projection; see \cite{AFW06}. Then using standard results in \cite{AFW06} it holds that $d(\Pi \omega)=0$ since $d \omega=0$, and moreover, if $k=n$, $\int_{T} \Pi \omega=\int_T \omega$.  
Therefore, $\xi \coloneqq \omega-\Pi \omega$ satisfies $d \xi=0$.
If $k\le n-1$, then there holds $\tr_F \xi=0$ for all $F \in \Delta_{n-1}(T)$,
whereas if $k=n$, then  $\int_T \xi =0$.  Thus $\xi\in \mathring{V}_{d,r}^k(T^\s)$,
and so, by Theorem \ref{mainthm}, there exists $\varphi  \in \mathring{M}_{r+1}^{k-1}(T^z)$   such that $d\varphi= \xi$.
Using the exact sequence property of  $\{\pol_r \La^k(T)\}$ there exists $\psi \in \pol_{r+1} \La^{k-1}(T)$ such that 
$d\psi= \Pi \omega$.  Setting $\rho = \varphi +\psi\in M^{k-1}_{d,r+1}(T^\s)$, we have $d\rho= \omega$.

Now consider the case $r=0$. Applying Lemma \ref{lemma101}, we define $\Pi\omega\in \pol_0\Lambda^k(T)$
uniquely by the conditions
\begin{equation*}
\int_f \tr_f \Pi \omega   =  \int_f \tr_f \omega  \quad \text{ for all } f \in S_k(T,x_0).
\end{equation*}
Let $\xi= \omega -\Pi \omega$ then by  Lemma \ref{lemma101} have that $\tr_{F_i} \xi= 0$ for $1 \le i \le n$. . Consider the case, $k \le n-1$. Consider an arbitrary $k$ sub-simplex   $f \in \Delta_k(F_0)$ then it also belongs to another $f \in \Delta_k(F_j)$ for another $j \neq 0$. Hence, $\tr_f \xi=0$. Applying Lemma  \ref{lemma101} once more we have that $\tr_{F_0} \xi=0$.  Hence, we have that
 $\tr_F \xi=0$ for all $F \in \Delta_{n-1}(T)$.  Moreover, if $k =n$ we see that $\int_{T} \xi =0$.  Therefore, $\xi \in \mathring{V}_{d,0}^k(T^\s)$.  Hence, using  Theorem \ref{mainthm} we have a $\rho \in \mathring{M}_{1}^{k-1}(T^\s)$   such that $d\rho= \xi$.  Again, using the exact sequence property of  $\{\pol_r \La^k(T)\}$ there exists $\psi \in \pol_{1} \La^{k-1}(T)$ such that 
$d\psi= \Pi \omega$, and  hence $d(\rho+ \psi)= \omega$. \hfill $\qed$


\subsection{Proof of Corollary \ref{cor2}}
Let $\omega \in \mathring{M}_{r}^k(T^z) \subset \mathring{V}_{d,r}^k(T^z)$ and $d\omega=0$. Theorem \ref{mainthm} gives a $\rho \in \mathring{M}_{r+1}^{k-1}(T^z)$ such that $d \rho=\omega$ and therefore $d \rho$ is continuous and vanishes on $\partial T$. In other words, $\rho \in \mathring{M}_{d,r+1}^{k-1}(T^z)$. \hfill $\qed$

\subsection{Proof of Corollary \ref{cor3}}
The result follows from Corollary \ref{cor2} by noting that $\mathring{M}_{d,r}^k(T^z)  \subset \mathring{M}_{r}^k(T^z)$.
\hfill $\qed$


\subsection{Proof of Corollary \ref{cor4}}
We consider the following sequences:
\begin{align*}
\begin{array}{ccccccccccccc}
&\cdots& \stackrel{d}{\xrightarrow{\hspace*{0.5cm}}}&
M_{d,r+1}^{k-1}(T^\s)\!\!\!
&\stackrel{d}{\xrightarrow{\hspace*{0.5cm}}}&
\!\!\! M^k_{d,r}(T^\s)
\!\!\!&\stackrel{d}{\xrightarrow{\hspace*{0.5cm}}}&
\!\!\!M^{k+1}_{r-1}(T^\s)
\!\!\!\!& 
\stackrel{d}{\xrightarrow{\hspace*{0.5cm}}}& 
\!\!\! V^{k+2}_{d,r-2}(T^\s)
&
\!\!\!\!& 
\stackrel{d}{\xrightarrow{\hspace*{0.5cm}}}& 
\cdots\\
&\cdots& \stackrel{d}{\xrightarrow{\hspace*{0.5cm}}}&
M_{d,r+1}^{k-1}(T^\s)\!\!\!
&\stackrel{d}{\xrightarrow{\hspace*{0.5cm}}}&
\!\!\! M^k_{r}(T^\s)
\!\!\!&\stackrel{d}{\xrightarrow{\hspace*{0.5cm}}}&
\!\!\!V^{k+1}_{d,r-1}(T^\s)
\!\!\!\!& 
\stackrel{d}{\xrightarrow{\hspace*{0.5cm}}}& 
\!\!\! V^{k+2}_{d,r-2}(T^\s)
&
\!\!\!\!& 
\stackrel{d}{\xrightarrow{\hspace*{0.5cm}}}& 
\cdots
\end{array}
\end{align*}
Theorem \ref{cor1} and the results
in \cite{AFW06} show that both sequences
are exact, i.e., the range of each map is the kernel
of the succeeding map.

Denote by
\begin{align*}
\ker M_{d,r}^k(T^\s) &= \{\omega\in M_{d,r}^k(T^\s):\ d\omega=0\},\\
\range  M_{d,r}^k(T^\s) & = \{d\omega:\ \omega\in M_{d,r}^k(T^\s)\}.
\end{align*}
The rank-nullity theorem shows that
\begin{align}\label{eqn:RNT}
\dim M_{d,r}^k(T^\s) = \dim \ker M_{d,r}^k(T^\s) +\dim \range M_{d,r}^k(T^\s),
\end{align}
and the exactness of the first sequence gives
\begin{align*}
\dim \range M_{d,r}^k(T^\s) 
& = \dim \ker M_{r-1}^{k+1}(T^\s)\\
&= \dim M_{r-1}^{k+1}(T^\s) - \dim \range M_{r-1}^{k+1}(T^\s) \\
&= \dim M_{r-1}^{k+1}(T^\s) - \dim \ker V_{d,r-2}^{k+2}(T^\s).
\end{align*}
On the other hand, we have, by the exactness of the second sequence,
\begin{align*}
\dim \ker M_{d,r}^k(T^\s) 
&= \dim \ker M_{r}^k(T^\s) = \dim M_r^k(T^\s) - \dim \range M_r^k(T^\s)\\
& = \dim M_r^k(T^\s) - \dim \ker V_{d,r-1}^{k+1}(T^\s)\\
& = \dim M_r^k(T^\s) - \big(\dim V_{d,r-1}^{k+1}(T^\s) - \dim \range V_{d,r-1}^{k+1}(T^\s)\big)\\
& = \dim M_r^k(T^\s) - \big(\dim V_{d,r-1}^{k+1}(T^\s) - \dim \ker V_{d,r-2}^{k+2}(T^\s)\big).
\end{align*}
Applying these identities to \eqref{eqn:RNT},
we find that
\begin{align*}
\dim M_{d,r}^k(T^\s) 
&= \dim M_{r-1}^{k+1}(T^\s) - \dim \ker V_{d,r-2}^{k+2}(T^\s) \\
&\quad + \dim M_r^k(T^\s) - \dim V_{d,r-1}^{k+1}(T^\s) + \dim \ker V_{d,r-2}^{k+2}(T^\s)\\
& = \dim M_{r-1}^{k+1}(T^\s) + \dim M_r^k(T^\s) - \dim V_{d,r-1}^{k+1}(T^\s).
\end{align*}
The dimension count \eqref{eqn:DimCount2} is obtained similarly.
This concludes the proof.
\hfill $\qed$


\section{Local Smooth Finite Element de Rham Complexes in three dimensions}\label{sec-Smooth3D}
In this section we translate some of the results of Section \ref{section-3}
 in three dimensions ($n=3$) using vector proxies. 
 Namely, we reprove the results using vector notation
and standard differential operators for the benefit of the readers
that are more comfortable with vector calculus notation.
Moreover, we define local de Rham complexes in three dimension 
with enhanced smoothness and provide unisolvent sets of degrees of freedom.   The last two spaces in 
one of the sequences
correspond to the divergence-free velocity and pressure Stokes elements developed by Zhang \cite{Zhang04}.
The complex we propose characterize the divergence-free subspace of the discrete velocity 
space as well as show the relationship between the Stokes pair and the $C^1$
Clough-Tocher element \cite{Alfeld84}.

We start by translating our spaces using vector notation
by identifying $0$- and $3$-forms with scalar functions,
and $1$- and $2$-forms with vector-valued functions.  With a slight abuse of notation,
we set
\begin{alignat*}{2}
V_r^3(T^\s)=V_r^0(T^\s)=&\{ \omega \in L^2(T):   \omega|_{T_i} \in \pol_r (T_i) \text{ for } 0 \le i \le 3 \}, \\
V_r^1(T^\s)= V_r^2(T^\s)= &[V_r^0(T^\s)]^3.
\end{alignat*}
We then define
\begin{alignat*}{1}
V_{d,r}^0(T^\s)=& \{ \omega \in V_r^0(T^\s): \grad \omega \in [L^2(T)]^3\}, \\ 
V_{d,r}^1(T^\s)=& \{ \omega \in V_r^1(T^\s): \curl \omega \in [L^2(T)]^3 \},\\
V_{d,r}^2(T^\s)=&  \{ \omega \in V_r^2(T^\s): \dive \omega \in L^2(T) \}, \\
V_{d,r}^3(T^\s)=& V_r^3(T^\s),
\end{alignat*}
where the differential operators appearing in these definitions
are understood to be in the weak sense.
We further set 
\begin{alignat*}{2}
M_r^3(T^\s)=M_r^0(T^\s)=&\{ \omega \in C^0(T):   \omega|_{T_i} \in \pol_r (T_i) \text{ for } 0 \le i \le 3 \}, \\
M_r^1(T^\s)= M_r^2(T^\s)= &[M_r^0(T^\s)]^3,
\end{alignat*}
and 
\begin{alignat*}{1}
M_{d,r}^0(T^\s)=& \{ \omega \in M_r^0(T^\s): \grad \omega \in [C^0(T)]^3\}, \\ 
M_{d,r}^1(T^\s)=& \{ \omega \in M_r^1(T^\s): \curl \omega \in [C^0(T)]^3 \},\\
M_{d,r}^2(T^\s)=&  \{ \omega \in M_r^2(T^\s): \dive \omega \in C^0(T) \}, \\
M_{d,r}^3(T^\s)=& M_{r}^3(T^\s).
\end{alignat*}
The spaces with  homogenous boundary conditions are given by
\begin{alignat*}{1}
\mathring{V}_{d,r}^0(T^\s)=& \{ \omega \in V_{d,r}^0(T^\s): \omega|_F=0 \text{ for all } F \in \Delta_{2}(T)\}, \\ 
\mathring{V}_{d,r}^1(T^\s)=& \{ \omega \in V_{d,r}^1(T^\s): \omega \times n_F|_F=0  \text{ for all } F \in \Delta_{2}(T)\},\\
\mathring{V}_{d,r}^2(T^\s)=&  \{ \omega \in V_{d,r}^2(T^\s): \omega \cdot n_F|_F=0, \text{ for all } F \in \Delta_{2}(T) \}, \\
\mathring{V}_{d,r}^3(T^\s)=&\{ \omega \in  V_{d,r}^3(T^\s): \int_T \omega\, \mathrm{dx}=0 \}.
\end{alignat*}
Here, $n_F$ is unit normal pointing out of $F$.

For $0 \le k  \le 2$ we define
\begin{alignat*}{1}
\mathring{M}_{r}^k(T^\s)=&\{ \omega \in M_{r}^k(T^\s): \omega|_F=0 \text{ for all } F \in  \Delta_{2}(T)\}, \\ 
\end{alignat*}
and for $k=3$,
\begin{alignat*}{1}
\mathring{M}_{r}^3(T^\s)=&\{ \omega \in M_{r}^3(T^\s): \omega|_F=0 \text{ for all } F \in  \Delta_{2}(T), \int_T \omega=0\}.  
\end{alignat*}
Finally, we define 
\begin{alignat*}{1}
\mathring{M}_{d,r}^0(T^\s)=& \{ \omega \in M_{d,r}^0(T^\s) \cap \mathring{M}_{r}^0(T^\s): \grad \omega|_F=0   \text{ for all } F \in \Delta_{2}(T)\}, \\ 
\mathring{M}_{d,r}^1(T^\s)=& \{ \omega \in M_{d,r}^1(T^\s) \cap \mathring{M}_{r}^1(T^\s): \curl \omega|_F=0  \text{ for all } F \in \Delta_{2}(T)\},\\
\mathring{M}_{d,r}^2(T^\s)=&  \{ \omega \in M_{d,r}^2(T^\s) \cap \mathring{M}_{r}^2(T^\s): \dive \omega |_F=0, \text{ for all } F \in \Delta_{2}(T) \}, \\
\mathring{M}_{d,r}^3(T^\s)=&\mathring{M}_{r}^3(T^\s).
\end{alignat*}

Note that $\mathring{V}_{d,r}^0(T^\s)$ is the $H^1_0$-conforming Lagrange finite element space, 
$\mathring{V}_{d,r}^1(T^\s)$ is the $H_0({\rm curl})$-conforming Nedelec
space of second type, and $\mathring{V}_{d,r}^2(T^\s)$ is the $H_0({\rm div})$-conforming
Nedelec space of the second type.  The space $\mathring{V}_{d,r}^3(T^\s)$ is simply
the space of piecewise polynomials with vanishing mean, but without any continuity restrictions.
Together, these canonical finite element spaces form an exact
discrete de Rham complex:
\begin{align}\label{eqn:FEMStandardDeRham}
\begin{array}{ccccccccccccc}
&0 & {\xrightarrow{\hspace*{0.1cm}}}&
\mathring{V}_{d,r}^0(T^\s)\!\!\!
&\stackrel{\grad}{\xrightarrow{\hspace*{0.5cm}}}&
\!\!\! 
\mathring{V}_{d,r-1}^1(T^\s)
\!\!\!&\stackrel{\curl}{\xrightarrow{\hspace*{0.5cm}}}&
\mathring{V}_{d,r-2}^2(T^\s)
\!\!\!\!& \stackrel{{\rm div}}{\xrightarrow{\hspace*{0.5cm}}}& 
\!\!\!
\mathring{V}_{d,r-3}^3(T^\s)
&
\!\!\!\xrightarrow{\hspace*{0.1cm}}  &0.
\end{array}
\end{align}

Theorems \ref{mainthm}--\ref{cor1} and Corollaries  \ref{cor2} -- \ref{cor4} hold with $d$ being one of the differential operators $\grad, \curl, \dive$. 
Essentially these results show that any discrete space in \eqref{eqn:FEMStandardDeRham}
can be replaced by its continuous analogue, and the exactness property will still be preserved
provided that the spaces to the left of the replacement are modified accordingly.  
In particular, Theorems \ref{mainthm}--\ref{cor2}
and Corollaries \ref{cor3}--\ref{cor4} show that
the following sequences are exact:
\begin{subequations}
\label{eqn:AllSequences}
\begin{alignat}{5}
\label{eqn:AllSequences1}
&0\
{\xrightarrow{\hspace*{0.5cm}}}\
\mathring{M}_{d,r}^0(T^\s)\
&&\stackrel{\grad}{\xrightarrow{\hspace*{0.5cm}}}\
\mathring{M}_{d,r-1}^1(T^\s)\
&&\stackrel{\curl}{\xrightarrow{\hspace*{0.5cm}}}\
\mathring{M}_{d,r-2}^2(T^\s)\
&&\stackrel{{\rm div}}{\xrightarrow{\hspace*{0.5cm}}}\
\mathring{M}_{d,r-3}^3(T^\s)\
&&\xrightarrow{\hspace*{0.5cm}}\
 0,\\
\label{eqn:AllSequences2}
&0 \
{\xrightarrow{\hspace*{0.5cm}}}\
\mathring{M}_{d,r}^0(T^\s)\
&&\stackrel{\grad}{\xrightarrow{\hspace*{0.5cm}}}\
\mathring{M}_{d,r-1}^1(T^\s)\
&&\stackrel{\curl}{\xrightarrow{\hspace*{0.5cm}}}\
\mathring{M}_{r-2}^2(T^\s)\
&&\stackrel{{\rm div}}{\xrightarrow{\hspace*{0.5cm}}}\
\mathring{V}_{d,r-3}^3(T^\s)\
&&\xrightarrow{\hspace*{0.5cm}}\
0,\\
\label{eqn:AllSequences3}
&0\
{\xrightarrow{\hspace*{0.5cm}}}\
\mathring{M}_{d,r}^0(T^\s)\
&&\stackrel{\grad}{\xrightarrow{\hspace*{0.5cm}}}\
\mathring{M}_{r-1}^1(T^\s)\
&&\stackrel{\curl}{\xrightarrow{\hspace*{0.5cm}}}\
\mathring{V}_{d,r-2}^2(T^\s)\
&&\stackrel{{\rm div}}{\xrightarrow{\hspace*{0.5cm}}}\
\mathring{V}_{d,r-3}^3(T^\s)\
&&\xrightarrow{\hspace*{0.5cm}}\  
0.
\end{alignat}
\end{subequations}
For example, the exactness of the third sequence means:
\begin{subequations}\label{vector}
\begin{alignat}{1}
&\text{ If } \omega \in \mathring{M}_{r-1}^1(T^\s) \text{ and } \curl \omega=0,  \text{ there exists } \rho \in \mathring{M}_{d,r}^0(T^\s)  \text{ such that }  
\grad \rho=\omega. \label{vector1} \\
&\text{ If } \omega \in \mathring{V}_{d,r-2}^2(T^\s) \text{ and } \dive \omega=0,  \text{ there exists } \rho \in \mathring{M}_{r-1}^1(T^\s)  \text{ such that }  
\curl \rho=\omega. \label{vector2} \\
& \text{ If } \omega \in \mathring{V}_{d,r-3}^3(T^\s),  \text{ there exists } \rho \in \mathring{V}_{d,r-2}^2(T^\s)  \text{ such that }  
\dive \rho=\omega.  \label{vector3}
\end{alignat}
\end{subequations}
Note that \eqref{vector2} is the main contribution among the three: Property \eqref{vector3}
follows from the exactness of \eqref{eqn:FEMStandardDeRham},
and if $\omega\in \mathring{M}_{r-1}^1(T^\s)\subset \mathring{V}_{d,r-1}^1(T^\s)$ satisfies $\curl \omega=0$,
then the exactness of \eqref{eqn:FEMStandardDeRham} implies that $\omega = \grad \rho$
for some $\rho\in V_{d,r}^0(T^\s)$.  By definition, we get $\rho\in \mathring{M}_{d,r}^0(T^\s)$,
i.e., property \eqref{vector1}.

Although we have proved these results in the previous section using the language of differential forms, 
we will give a sketch a proof \eqref{vector2} using vector notation for 
the benefit of those readers that feel more comfortable with vector notation.   We start by giving an instance of Lemma \ref{lemma1}.

\begin{lemma}\label{lemma111}
Let $T=[x_0, x_1, x_2, x_3]$. Suppose that $\omega \in [\pol_r(T)]^3$ with $\omega \cdot n_{F_i}=0$ on $F_i$. Then 
\begin{equation*}
\omega= \grad \la_i \times v+\la_i w,
\end{equation*}
where $v \in [\pol_r(T)]^3$ and  $w \in [\pol_{r-1}(T)]^3$.
\end{lemma}
\begin{proof}
With out loss of generality we assume that $i=3$. Then it is easy to see that 
\begin{equation}\label{eqn:omega3DExpansion}
\omega=a_1\, \grad \la_2 \times \grad \la_3+a_2 \, \grad \la_1 \times \grad \la_3+  a_3 \, \grad \la_1 \times \grad \la_2,
\end{equation}
where $a_1, a_2, a_3 \in \pol_r(T)$. Since $\omega \cdot n_{F_3}=0$ on $F_3$
and $\grad \la_3$ is parallel to $n_{F_3}$, we have  $\omega \cdot \grad \la_3=0$ on $F_3$.
Applying this identity to \eqref{eqn:omega3DExpansion}, we have
\begin{equation*}
a_3  (\grad \la_1 \times \grad \la_2) \cdot \grad \la_3=0 \qquad \text{ on } F_3.
\end{equation*}
This implies that $a_3=0$ on $F_3$, or equivalently, that $a_3=\la_3 b$ for some  $b \in \pol_{r-1}(T)$. The result now follows if we let 
$v=-a_1 \, \grad \la_2-a_2\, \grad \la_1$ and $w= b\,  \grad \la_1 \times \grad \la_2$.
\end{proof}

Next we state an instance of Lemma \ref{lemma2}.
\begin{lemma}
Any $\omega \in \mathring{V}_{d,r}^2(T^\s)$ satisfies 
\begin{equation}\label{rep0}
\omega= \grad \m \times v+ \m  w
\end{equation}
for some $v \in V_r^{1}(T^\s)$ and $w \in  V_{r-1}^{2}(T^\s)$. Moreover, $v \cdot t_e$ is single-valued on all edges of $e$ of $T$, where
 $t_e$ is a unit tangent vector to $e$.
\end{lemma}
\begin{proof}
First by Lemma \ref{lemma111} we see that $\omega$ has the form \eqref{rep0}. Thus, to complete the proof, we must show that $v \cdot t_e$ is single-valued on all edges of $e$ of $T$. 

To this end,  let $e$ be an edge of $T$. Let $T_1,T_2 \in T^\s$ such that they have a common (internal) face $F$ and that 
$e$ is an edge of the face $F$. Let $n$ be a unit normal vector to the face $F$. Since the tangential components of $\grad \m$ on $F$ are single-valued we have that
\begin{equation}\label{aux316}
\grad \m|_{T_1}= \grad \m|_{T_2} +a n
\end{equation}
for a constant $a$.  Since $\omega \in \mathring{V}_{d,r}^2(T^\s)$ it must be that
\begin{equation*}
\omega|_{T_1} \cdot n=\omega|_{T_2} \cdot n \qquad \text{ on } F.
\end{equation*}
In particular, if we use that $\mu=0$ on $e$, we have
\begin{equation*}
(\grad \m |_{T_1} \times v|_{T_1}) \cdot n=(\grad \m|_{T_2} \times v |_{T_2}) \cdot n \quad \text{ on } e.
\end{equation*}
Therefore,
\begin{equation*}
(\grad \m |_{T_1} \times n ) \cdot v|_{T_1} =(\grad \m|_{T_2} \times n) \cdot v |_{T_2}  \quad \text{ on } e.
\end{equation*}
By \eqref{aux316},  $(\grad \m |_{T_1} \times n ) =(\grad \m|_{T_2} \times n)$ which is parallel  to $t_e$. This proves the result.  
\end{proof}

We now prove an instance of Lemma \ref{crucial}.
\begin{lemma}\label{lemmacrucial}
Let $\omega \in \mathring{V}_{d,r}^2(T^\s)$ and  let $\ell \ge 0$ be an integer. There exists  $\gamma \in [\pol_{r}(T)]^3$ and $\psi \in V_{d,r-1}^2(T^\s)$  (in the case $r=0$, $\psi  \equiv 0$) such that 
\begin{equation}\label{eq1111}
\m^{\ell} \omega=\curl (\m^{\ell+1} \gamma)+ \m^{\ell+1} \psi.
\end{equation}
\end{lemma}

\begin{proof}
We prove the case $r \ge 1$ and leave the case $r=0$ to the reader.
By the previous lemma we have
\begin{equation*}
\omega= \grad \m \times v+ \m  w
\end{equation*}
for some $v \in V_r^{1}(T^\s)$ and $w \in  V_{r-1}^{2}(T^\s)$. Moreover, $v \cdot t_e$ is single-valued on all edges of $e$ of $T$. 
Applying Proposition \ref{Prop2}, we uniquely define  $\gamma \in [\pol_{r}(T)]^3$ such that it satisfies
\begin{subequations}
\begin{alignat}{4}
\label{eqn:Local3D123A}
(\ell+1) \int_e (\gamma \cdot t_e)\eta\, \mathrm{ds}  =& \int_e (v \cdot t_e) \eta\,\mathrm{ds}, \quad &&\forall \eta\in \pol_r(e),\quad &&\forall e\in \Delta_1(T),\\
\label{eqn:Local3D123B}
(\ell+1) \int_F (\gamma \times n_F)   \cdot \eta\,\mathrm{dA} =&\int_F (v \times n_F) \cdot \eta\,\mathrm{dA}, \quad &&\forall \eta\in D_{r-1}(F),\quad &&\forall F\in \Delta_2(T),\\
\label{eqn:Local3D123C}
\int_T \gamma \cdot \eta\,\mathrm{dx} =& 0, \quad &&   \forall \eta \in D_{r-2}(T).
\end{alignat}
\end{subequations}
Here, $D_{s}(F)= \pol_{s-1}(F)+ x_F \pol_{s-1}(F)$ is the local Raviart-Thomas space on $F$,
 and $D_{s}(T)= \pol_{s-1}(T)+ x \pol_{s-1}(T)$ is the local Raviart-Thomas space on $T$.
Using \eqref{eqn:Local3D123A}--\eqref{eqn:Local3D123B} and Stokes Theorem, we easily find 
that $(\ell+1) \gamma \times n_F= v\times n_F$ on $F$ for all faces $F$ of $T$. Because
$\grad \mu$ is parallel to $n_F$, we have
\begin{equation*}
\grad \m \times v= (\ell+1) \grad \m \times \gamma+ \m \phi,
\end{equation*}
for some $\phi \in V_{r-1}^2(T^\s)$.  Using the product rule we get
\begin{equation*}
\curl (\m^{\ell+1} \gamma)=(\ell+1) \m^\ell (\grad \m \times \gamma)+ \mu^{\ell+1} \curl \gamma, 
\end{equation*}
and hence
\begin{equation*}
\m^\ell( \grad \m \times v)=
\curl (\m^{\ell+1} \gamma)+ \m^{\ell+1}( -\curl \gamma+ \phi). 
\end{equation*}
If we let $\psi=-\curl \gamma+ \phi+ w \in V_{r-1}^2(T^\s)$ we arrive at the equation \eqref{eq1111}.  We know that $\mu^\ell \omega \cdot n$ and $d(\m^{\ell+1}  \gamma) \cdot n$ are single valued across all interior face of $T^\s$, and therefore, $\mu^{\ell+1} \psi \cdot n$ is also single valued. Hence, $\psi \in V_{d,r-1}^2(T^\s)$.
\end{proof}

\begin{proof}[Proof of \eqref{vector2}.]  For readability, we prove \eqref{vector2} with $r-1$ replaced by $r$.

Let $\omega \in \mathring{V}_{d,r}^2(T^\s)$ and $\Div \omega=0$. Assume we have found $ \gamma_{r}, \ldots, \gamma_{r-j}$
with $\gamma_{\ell} \in [\pol_{\ell}(T)]^3$  
and $\omega_{r-(j+1)} \in  V_{d,r-(j+1)}^2(T^\s)$ such that 
\begin{equation*}
\omega=\curl(\m \gamma_r+ \m^2 \gamma_{r-1}+ \cdots+\m^{j+1} \gamma_{r-j}) + \m^{j+1} \omega_{r-(j+1)}.
\end{equation*}
 Then, we see that 
 \begin{equation*}
 0=\dive (\m^{j+1} \omega_{r-(j+1)})= \m^j (\m\, \dive \omega_{r-(j+1)} + (j+1) \grad \m \cdot \omega_{r-(j+1)}),
 \end{equation*}
 which implies that $\m \, \dive \omega_{r-(j+1)} + (j+1) \grad \m \cdot \omega_{r-(j+1)}=0$ on $T$. Hence, we have that $\grad \m \cdot \omega_{r-(j+1)}=0$ on $\partial T$. This shows that $\omega_{r-(j+1)}\cdot n_F=0$ for all $F \in \Delta_{2}(T)$. Or in other words, $\omega_{r-(j+1)} \in  \mathring{V}_{d,r-(j+1)}^2(T^\s)$.
We then apply Lemma \ref{lemmacrucial} to get
\begin{equation*}
\m^{j+1} \omega_{r-(j+1)}=\curl (\m^{j+2} \gamma_{r-(j+1)})+ \m^{j+2} \omega_{r-(j+2)},
\end{equation*} 
where  $\gamma_{r-(j+1)} \in [\pol_{r-(j+1)}(T)]^3$ and $\omega_{r-(j+2)} \in V_{d,r-(j+2)}^2(T^\s)$. 
It follows that
\begin{equation*}
\omega=\curl(\m \gamma_r+ \m^2 \gamma_{r-1}+ \cdots+\m^{j+1} \gamma_{r-j}+ \m^{j+2} \gamma_{r-(j+1)}) + \m^{j+2} \omega_{r-(j+2)}.
\end{equation*}
Continuing by induction we have
\begin{equation*}
\omega=\curl(\m \gamma_r+ \m^2 \gamma_{r-1}+ \cdots+\m^{r} \gamma_{1}+ \m^{r+1} \gamma_0).
\end{equation*}
This completes the proof.
\end{proof}

\subsection{Degrees of Freedom}\label{subsec-DOFs}

Our goal is to develop degrees of freedom (DOFs), and in turn, to 
construct analogous global versions of the spaces appearing in \eqref{eqn:AllSequences}.
However, to develop DOFs, special care must be taken
to ensure that the induced finite element spaces satisfy the same
exactness properties as \eqref{eqn:AllSequences} due to the intrinsic smoothness of
the spaces. In particular, 
it is a simple exercise (cf. \cite{Alfeld84}) to show that functions in $M_{d,r}^0(T^\s)$
are $C^2$ at the vertices in $T^\s$, and this influences the construction of DOFs
and global finite element spaces.   For example, if we consider
the global analogue of the third sequence in \eqref{eqn:AllSequences}, then 
natural choices would to take 
${M}_{r-1}^1$ as the vector-valued Lagrange space, $V_{d,r-2}^2$
 the $H({\rm div})$-conforming Nedelec space, and $V_{d,r-3}^3$ 
 the space of piecewise polynomials.  This selection would indeed form
 a discrete (global) complex, but a simple counting argument shows that the resulting
 sequence is {\em not} exact on general contractible domains.

To construct the desired global spaces, it seems necessary
to consider finite element spaces with additional smoothness
at the vertices.
In particular, guided by the $C^1$ Clough-Tocher space, 
we consider the subspaces of $M_{d,r-k}^k(T^\s)$, $M_{r-k}^k(T^\s)$,
and $V_{d,r-k}^k(T^\s)$ that have $C^{2-k}$ continuity on $\Delta_0(T^\s)$,
and formulate the global finite element spaces using these subspaces
(in the case $k=3$, no additional continuity is added).
This framework is also adopted in \cite{Christiansen18} on general meshes, where
finite element spaces are constructed that form a subsequence
of the de Rham complex with minimal $L^2$ smoothness.
Here, we show that, on Alfeld splits, this framework yields
finite element spaces with greater global smoothness.

However, it turns out that these additional smoothness constraints at the vertices
are redundant in many cases as the next lemma shows.  Its proof
is given in the appendix.

\begin{lemma}\label{lem:SmoothAtVerts}
Any $\omega \in M_{d,r}^k(T^\s)$ is $C^{2-k}$ on $\Delta_0(T^\s)$ for $k=0,1,2$.
\end{lemma}

We introduce the local spaces 
with added continuity at the vertices as
\begin{align*}
\Mrc(T^\s):=\{\omega\in M_{r-1}^1(T^\s):\ \omega\text{ is $C^{1}$ on $\Delta_0(T^\s)$}\},\\
\Vrc(T^\s):=\{\omega\in V_{d,r-2}^2(T^\s):\ \omega\text{ is $C^{0}$ on $\Delta_0(T^\s)$}\},
\end{align*}
and set $\Mrco(T^\s) = \Mrc(T^\s)\cap \mathring{M}_{r-1}^1(T^\s)$
and $\Vrco(T^\s) = \Vrc(T^\s)\cap \mathring{V}_{r-2}^2(T^\s)$.

\begin{remark}\label{rem:StenbergCounting}
The space $\Vrc(T^\s)$ corresponds
to the nodal $H({\rm div})$ finite element
introduced in \cite{Stenberg10,Christiansen18}, and
the space $\Mrc(T^\s)$ is a vector-valued Hermite
finite element space.    Using \cite[Lemma 10]{Christiansen18}, 
we have for $r\ge 4$
\begin{align*}
\dim \Vrc(T^\s)
& = 3(\# \Delta_0(T^\s))+\big(\frac12 r(r-1)-3\big)(\#\Delta_2(T^\s))+\frac12r(r-3)(r-1)(\#\Delta_3(T^\s))\Big]\\
&=(2r-5)(r^2+r+3),\\
\dim \Vrco(T^\s)
& = 3(\# \Delta_0(T^\s)\backslash \Delta_0(T))+\big(\frac12 r(r-1)-3\big)(\#\Delta_2(T^\s)\backslash \Delta_2(T))\\
&\qquad+\frac12r(r-3)(r-1)(\#\Delta_3(T^\s))\Big]\\
%
&=2r^3-5r^2+3r-15.
\end{align*}
%
%
Furthermore, using \cite[Lemma 5]{Christiansen18} we find that ($r\ge 4$)
\begin{align*}
\dim \Mrc(T^\s)
& = 3\Big[ 4(\# \Delta_0(T^\s))+(r-4)(\#\Delta_1(T^\s))+\frac12(r-2)(r-3)(\#\Delta_2(T^\s))\\
&\qquad+\frac16(r-2)(r-3)(r-4)\#\Delta_3(T^\s)\Big]\\
&=(r-2)(2r^2+r+9),\\
\dim \Mrco(T^\s)
& = 3\Big[ 4(\# \Delta_0(T^\s)\backslash \Delta_0(T))+(r-4)(\#\Delta_1(T^\s)\backslash \Delta_1(T))+\frac12(r-2)(r-3)(\#\Delta_2(T^\s)\backslash \Delta_2(T))\\
&\qquad+\frac16(r-2)(r-3)(r-4)\#\Delta_3(T^\s)\Big]\\
%
%
& = (r-3)(2r^2-3r+10).
\end{align*}
\end{remark}

\begin{lemma}\label{lem:VcoMcoLow}
If $r\le 3$, then $\Mrc(T^\s) = [\pol_{r-1}(T)]^3$ and $\Vrc(T^\s) = M_{r-2}^2(T^\s)$.
In particular, the above dimension counts for $\Mrc(T^\s)$ and $\Vrc(T^\s)$
are valid in the case $r=3$ as well.
\end{lemma}
\begin{proof}
We consider the case $r=3$, as the other cases are considerably simpler.

If $r=3$, then $\Vrc(T^\s)$ consists of vector-valued piecewise linear
polynomials that are continuous at the vertices.  This implies that the functions
are continuous on $T$, and thus $\Vrc(T^\s) = M_{r-2}^2(T^\s)$.

Next, we write
\[
M_{2}^1(T^\s) = [\pol_2(T)]^3+ \mu [\pol_1(T)]^3+ \mu^2 [\pol_0(T)]^3,
\]
and note the inclusion $M_{c,2}^1(T^\s)\subset M_2^1(T^\s)$.
Let $\omega\in M_{c,2}^1(T^\s)$ and write
$\omega = \omega^{(2)} + \mu \omega^{(1)}+\mu^2 \omega^{(0)}$ with $\omega^{(i)}\in [\pol_i(T)]^3$.
Let $T_1,T_2\in T^\s$ and set $F = \p T_1\cap \p T_2$.  We then have
\begin{align}
\label{eqn:ProductRule156}
\grad (\omega_1-\omega_2) = \omega^{(1)}\otimes (\grad \mu_1-\grad \mu_2) +2 \mu \omega^{(0)}\otimes (\grad \mu_1-\grad\mu_2)\text{  on $F$},
\end{align}
where $\omega_j$ and $\mu_j$ is the restriction of $\omega$ and $\mu$ to $T_j$, respectively.
Restricting this identity to the boundary vertex $a\in \Delta_0(F)\cap \Delta_0(T)$, we obtain
\[
0 = \grad (\omega_1-\omega_2)(a) = \omega^{(1)}(a)\otimes (\grad \mu_1-\grad \mu_2),
\]
which implies that $\omega^{(1)}(a)=0$.  Thus, $\omega^{(1)}$ vanishes on all vertices of $T$,
and therefore $\omega^{(1)}\equiv 0$.

Next, we restrict \eqref{eqn:ProductRule156} to the barycenter of $T$ to get
\begin{align*}
0 = \grad (\omega_1-\omega_2)(z) =2 \omega^{(0)}(z)\otimes (\grad \mu_1-\grad\mu_2).
\end{align*}
We then conclude that $\omega^{(0)} \equiv 0$, and so $\omega = \omega^{(2)}\in [\pol_2(T)]^3$.
Thus, $M_{c,2}^1(T^\s) = [\pol_2(T)]^3$.
\end{proof}

We study finite element spaces in the sequence, where $C^{2-k}$ continuity
on $\Delta_0(T^\s)$ is added to \eqref{eqn:AllSequences} and without boundary conditions:
\begin{subequations}
\label{eqn:AllSequencesC2}
\begin{alignat}{5}
\label{eqn:AllSequencesC21}
&\bbR\
{\xrightarrow{\hspace*{0.5cm}}}\
{M}_{d,r}^0(T^\s)\
&&\stackrel{\grad}{\xrightarrow{\hspace*{0.5cm}}}\
{M}_{d,r-1}^1(T^\s)\
&&\stackrel{\curl}{\xrightarrow{\hspace*{0.5cm}}}\
{M}_{d,r-2}^2(T^\s)\
&&\stackrel{{\rm div}}{\xrightarrow{\hspace*{0.5cm}}}\
{M}_{d,r-3}^3(T^\s)\
&&\xrightarrow{\hspace*{0.5cm}}\
 0,\\
\label{eqn:AllSequencesC22}
&\bbR  
{\xrightarrow{\hspace*{0.5cm}}}\
{M}_{d,r}^0(T^\s)\
&&\stackrel{\grad}{\xrightarrow{\hspace*{0.5cm}}}\
{M}_{d,r-1}^1(T^\s)\
&&\stackrel{\curl}{\xrightarrow{\hspace*{0.5cm}}}\
{M}_{r-2}^2(T^\s)\
&&\stackrel{{\rm div}}{\xrightarrow{\hspace*{0.5cm}}}\
{V}_{d,r-3}^3(T^\s)\
&&\xrightarrow{\hspace*{0.5cm}}\
0,\\
\label{eqn:AllSequencesC23}
&\bbR\
{\xrightarrow{\hspace*{0.5cm}}}\
{M}_{d,r}^0(T^\s)\
&&\stackrel{\grad}{\xrightarrow{\hspace*{0.5cm}}}\
\Mrc(T^\s)\
&&\stackrel{\curl}{\xrightarrow{\hspace*{0.5cm}}}\
\Vrc(T^\s)\
&&\stackrel{{\rm div}}{\xrightarrow{\hspace*{0.5cm}}}\
{V}_{d,r-3}^3(T^\s)\
&&\xrightarrow{\hspace*{0.5cm}}\  
0.
\end{alignat}
\end{subequations}
Note that, compared to \eqref{eqn:AllSequences},
only the second and third spaces in \eqref{eqn:AllSequencesC23} have been altered.

\begin{theorem}\label{thm:AllSequencesC2Exact}
The sequences \eqref{eqn:AllSequencesC2}
are exact for $r\ge 1$.
\end{theorem}
\begin{proof}
It has already been established that
the first two sequences \eqref{eqn:AllSequencesC21} and \eqref{eqn:AllSequencesC22} are exact.
We now show that \eqref{eqn:AllSequencesC23} is exact as well.

(i) The surjectivity ${\dive}:{V}_{c,r-2}^2(T^\s)\to V_{d,r-3}^3(T^\s)$ follows from the surjectivity of  ${\dive}:{M}_{r-2}^2(T^\s)\to V_{d,r-3}^3(T^\s)$ and the fact that ${M}_{r-2}^2(T^\s) \subset \Vrc(T^\s)$.


(ii) Similarly the surjectivity of ${\rm grad}:M_{d,r}^0(T^\s) \rightarrow  \ker\,\Mrc(T^\s)$ 
follows from the surjectivity of ${\rm grad}:M_{d,r}^0(T^\s) \rightarrow  \ker\, M_{r-1}^1(T^\s)$  
and the inclusion $\Mrc(T^\s) \subset M_{r-1}^1(T^s)$.


(iii) Let $r\ge 3$.  By the rank-nullity theorem, part (i), and Remark \ref{rem:StenbergCounting},
\begin{align*}
\dim \ker \Vrc(T^\s) 
&= \dim \Vrc(T^\s) - \dim V_{d,r-3}^3(T^\s)\\
& =(2r-5)(r^2+r+3) - \frac46 r(r-1)(r-2)\\
& =\frac43 r^3-r^2-\frac13 r-15.
\end{align*}
On the other hand, we have, by part (ii), Remark \ref{rem:StenbergCounting},
and Corollary \ref{cor4},
\begin{align*}
\dim \range \Mrc
%
& =\dim \Mrc(T^\s) - \dim \grad {M}_{d,r}^0(T^\s)\\
& =\dim \Mrc(T^\s) - \dim  {M}_{d,r}^0(T^\s)+1\\
& = (r-2)(2r^2+r+9) - \big(\frac23 r^3-2r^2+\frac{22}3r-2\big)+1\\
& = \frac43 r^3-r^2-\frac13 r-15.
\end{align*}
Thus, $\dim \ker \Vrc(T^\s)  =  \dim \range \Mrc(T^\s)$,
and since $ \range \Mrc\subset \ker \Vrc(T^\s)$
we conclude that $ \range \Mrc(T^\s) = \ker \Vrc(T^\s)$.  
We then conclude that the sequence is exact for $r\ge 3$.

If $r\le 2$, then $M_{d,r}^0(T^\s) = \pol_r(T)$, $\Mrc(T^\s) = [\pol_{r-1}(T)]^3$,
and $\Vrc(T^\s) = [\pol_{r-2}(T)]^3$, and so the sequence is clearly exact in this case.
\end{proof}

We now present  unisolvent sets of degrees
of freedom for the three-dimensional spaces 
in the previous section, and show that
the DOFs naturally induce a set of commutative projections.
First, we consider the family spaces with the largest amount of smoothness, $M_{d,r-k}^k(T^\s)$ $(k=0,1,2,3$).

Applying Corollary \ref{cor4}, we find that the dimension of these spaces are
\begin{alignat*}{2}
\dim M_{d,r}^0(T^\s) 
& = \frac23 r^3-2r^2+\frac{22}3r-2,\quad
&&\dim   M_{d,r-1}^1(T^\s)= (r-1)(2r^2-7r+18),\\
 \dim M_{d,r-2}^2(T^\s) &= \max\{2r^3-12r^2+32r-30,0\},\quad
&& \dim M_{d,r-3}^3(T^\s)   = \max\{\frac23 r^3-5r^2+\frac{43}{3}r -15,0\}.
\end{alignat*}
Likewise, we have
\begin{alignat*}{2}
\dim \mathring{M}_{d,r}^0(T^\s) 
& = \max\{\frac23(r-2))(r-3)(r-4),0\},\quad
&&\dim   \mathring{M}_{d,r-1}^1(T^\s)= \max\{(2r-5)(r-3)(r-4),0\},\\
 \dim \mathring{M}_{d,r-2}^2(T^\s) &= \max\{2(r-4)(r^2-6r+10),0\},\quad
&& \dim \mathring{M}_{d,r-3}^3(T^\s)   = \max\{\frac13 (2r-7)(r^2-7r+15)-1,0\}.
\end{alignat*}

We start with the DOFs of the $C^1$ finite element space;
the proof of the lowest order case ($r=5$) is found in \cite{Alfeld84,SplineBook}.
%
\begin{lemma}
 \label{lem:C1-DOF}
 Let $r\ge 5$. Then, a function  $\omega \in M_{d,r}^0(T^\s)$ is uniquely determined by the following degrees of freedom (DOFs):
\begin{subequations}
\label{eqn:SigmaDOFs}
\begin{alignat}{4}
\label{eqn:C1DOF1}
&D^\alpha \omega (a),\qquad && \forall |\alpha|\le 2,\quad &&\forall a\in \Delta_0(T)\qquad &&\text{($40$ DOFs)},\\
\label{eqn:C1DOF21}
&\int_e \omega \,\sigma \mathrm{ds},\qquad &&\forall \sigma\in \pol_{r-6}(e),\quad &&\forall e\in\Delta_1(T)\qquad &&
\text{($6(r-5)$ DOFs)},\\
\label{eqn:C1DOF22}
&\int_e \frac{\p \omega}{\p n_e^\pm}\,\sigma\,\mathrm{ds},\qquad 
&&\forall \sigma\in \pol_{r-5}(e),\quad &&\forall e\in\Delta_1(T)\qquad &&
\text{($12(r-4)$ DOFs)},\\
\label{eqn:C1DOF31}
&\int_F \omega \, \sigma \,\mathrm{dA},
\qquad 
&&\forall \sigma\in \pol_{r-6}(F),\quad &&\forall F\in\Delta_2(T)\qquad &&
\text{($4\frac{(r-5)(r-4)}{2}$ DOFs)},\\
\label{eqn:C1DOF32}
&\int_F \frac{\p \omega}{\p n_F}\, \sigma \,\mathrm{dA},
\qquad 
&&\forall \sigma\in \pol_{r-4}(F),\quad &&\forall F\in\Delta_2(T)\qquad &&
\text{($4\frac{(r-3)(r-2)}{2}$ DOFs)},\\
\label{eqn:C1DOF4}
&
\int_T {\grad \omega \cdot\grad\sigma}\,\mathrm{dx},
\qquad 
&&\forall \sigma\in \mathring{M}_{d,r}^{0}(T^\s),&& &&
\text{($2\frac{(r-4)(r-3)(r-2)}{3}$ DOFs)}.
\end{alignat}
\end{subequations}
Here, $n_{e_{\pm}}$ are two orthonormal normal vectors that are orthogonal to the edge $e$.
In the case $r = 6$, the sets listed in 
\eqref{eqn:C1DOF21} and \eqref{eqn:C1DOF31} are omitted.
\end{lemma}
\begin{proof}
 By a simple dimension count, we have the number of total DOFs
 in \eqref{eqn:SigmaDOFs} equals the dimension of the space $M_{d,r}^0(T^\s)$.

Let $\omega \in M_{d,r}^0(T^\s)$ be such that the DOFs \eqref{eqn:SigmaDOFs}
vanish.
The DOFs in \eqref{eqn:C1DOF1}--\eqref{eqn:C1DOF22} implies that 
$\omega|_e = 0$ and $\grad \omega|_e = 0$ for all $e\in \Delta_1(T)$.
 Combining these results with the DOFs on \eqref{eqn:C1DOF31}--\eqref{eqn:C1DOF32}, we conclude that 
 $\omega|_F = 0$ and $\grad \omega|_F = 0$ for all $F\in \Delta_2(T)$.
 Hence, $\omega \in \mathring{M}_{d,r}^{0}(T^\s)$. Taking $\sigma = \omega$ in \eqref{eqn:C1DOF4}, we get $\grad \omega=0$.
 Hence $\omega = 0$. This completes the proof.
\end{proof}

\begin{remark}
The set of DOFs is not unique. For example, 
we can obtain another set of DOFs by simply  
changing the internal DOFs \eqref{eqn:C1DOF4} in the set \eqref{eqn:SigmaDOFs} to be 
\[
 \int_T \omega \,\sigma\,\mathrm{dx},\quad \forall \sigma\in \mathring{M}_{d,r}^{0}(T^\s).
\]
The reason for our choice of DOFs \eqref{eqn:SigmaDOFs} will be clear in the next section when we discuss commutative projections.
\end{remark}

\begin{remark}\label{rem:H2Remark}
The proof of Lemma \ref{lem:C1-DOF}
shows that if $\omega\in M^0_{d,r}(T^\s)$ vanishes
at the DOFs \eqref{eqn:C1DOF1}--\eqref{eqn:C1DOF32}
restricted to a single face $F\in \Delta_2(T)$, then 
$\omega|_F=0$ and $\grad \omega|_F=0$.  Thus, 
the DOFs induce a global $C^1$ finite element space.
\end{remark}

\begin{lemma}\label{lem:H1curlDOFs}
Let $r\ge 5$, then
a function $\omega \in M_{d,r-1}^1 (T^\s)$ is uniquely determined by
the following degrees of freedom:
\begin{subequations}
\label{UDOFs}
\begin{alignat}{4}
\label{UDOFs1}
&D^\alpha \omega(a)\qquad &&\forall |\alpha|\le 1,
&& \forall a\in \Delta_0(T)\quad &&(\text{$48$ DOFs}),\\
\label{UDOFs2}
&\int_e \omega \cdot\kappa\,\mathrm{ds}\quad &&\forall \kappa\in [\pol_{r-5}(e)]^3,\quad &&\forall e\in \Delta_1(T)\quad &&
\text{($18(r-4)$ DOFs)},
\\
\label{UDOFs3}
&\int_e (\curl \omega)\cdot \kappa\, \mathrm{ds}\quad
&&\forall \kappa\in [\pol_{r-4}(e)]^3,\quad &&\forall e\in \Delta_1(T)\quad &&
\text{($18(r-3)$ DOFs)},\\
\label{UDOFs4}
&\int_f (\omega \cdot n_F) \kappa\, \mathrm{dA} \quad &&\forall \kappa\in \pol_{r-4}(F),\quad &&\forall F\in \Delta_2(T)\qquad &&\text{($2(r-2)(r-3)$ DOFs)}\\
\label{UDOFs5}
&\int_F (n_F \times \omega \times n_F) \cdot \kappa\, \mathrm{dA} \quad &&\forall \kappa\in {D}_{r-5}(F),\quad &&\forall F\in \Delta_2(T)\qquad &&\text{($4(r-3)(r-5)$ DOFs)},\\
\label{UDOFs6}
&\int_F (\curl \omega \times n_F) \cdot \kappa\, \mathrm{dA} \quad &&\forall \kappa\in [\pol_{r-5}(F)]^3, \quad &&\forall F\in \Delta_2(T)\quad &&\text{($4(r-3)(r-4)$ DOFs)},\\
\label{UDOFs7}
&\int_T \omega \cdot \kappa \, \mathrm{dx} \quad &&\forall \kappa\in \grad \mathring{M}_{d,r}^{0}(T^\s), && &&
\hspace*{-0.95cm}\text{($\frac{2(r-4)(r-3)(r-2)}{3}$ DOFs)},\\
\label{UDOFs8}
&\int_T \curl \omega \cdot \kappa \, \mathrm{dx} \quad &&\forall \kappa\in 
\curl\mathring{M}_{d,r-1}^1(T^\s), && &&
\hspace*{-0.95cm}\text{($\frac{(r-4)(r-3)(4r-11)}{3}$ DOFs)},
\end{alignat}
\end{subequations}
where we recall that $D_{r-5}(F)$ is the local Raviart--Thomas space
on the face $F$.
\end{lemma}
\begin{proof}
The number of conditions is $(r-1)(2r^2-7r+18)$ which equals the dimension of $M_{d,r-1}^1(T^\s)$.
We show that if $\omega \in M_{d,r-1}^{1}(T^\s)$ vanishes
at \eqref{UDOFs}, then $\omega \equiv 0$. In this case, it is easy to see that $\omega = \curl \omega =0$ on all edges,
and that $\curl \omega \times n_F =0$ and $\omega \cdot n_F=0$ on all faces.

To simplify notation, we use the following standard surface differential operators on a face $F$, with normal direction $n_F$ and 
tangential direction $t_F$ on its boundary $\partial F$:
For a smooth scalar field $\phi$, we denote
\begin{align*}
 {\rm grad}_F \phi =&\; n_F \times  \grad \phi \times n_F, \quad 
 { \text{rot}}_F \phi =\; \grad \phi \times n_F,
\end{align*}
and for a smooth vector field $\psi$, we denote
\begin{align*}
 {\rm curl}_F \psi =&\; n_F \cdot {\curl}\,  \psi, \quad
 {\rm div}_F \psi =\;  n_F\cdot \curl(n_F\times \psi).
\end{align*}
We also denote the tangential trace of a smooth vector field $\psi$ on $F$ as 
\[
 \psi_F = n_F \times \psi \times n_F.
\]

Stokes theorem on a face $F\subset \p T$ yields
\begin{align*}
\int_F ({\rm curl}_F \omega) q\, \mathrm{dA}- 
\int_F ({\text{rot}}_F q) \cdot \omega\, \mathrm{dA}=
\int_{\p F} \omega \cdot t q\, \mathrm{ds}
=0.
\end{align*}

For any $q \in  \pol_{r-5}(T)$, ${\rm grad}_F q \in D_{r-5}(F)$ and hence using\eqref{UDOFs5}, we have 
\begin{alignat*}{1}
\int_F ({\rm curl}_F \omega) q\,\mathrm{dA}=&\;
\int_F ({\text{rot}}_F q) \cdot \omega\, \mathrm{dA}\\
=&\;\int_F {\grad} q \cdot (n_F \times \omega)\, \mathrm{dA} \\
=&\; \int_F ({\grad} q \times n_F) \cdot 
(n_F \times \omega \times n_F)\, \mathrm{dA} = 0, \quad  \forall q\in \pol_{r-5}(T).
\end{alignat*}
Since ${\rm \curl}_F \omega \in \pol_{r-2}(F)$ and vanishes on the $\p F$  we conclude that ${\rm curl}_F \omega=0$ on $F$ and therefore 
$\curl \omega \cdot n_F ={\rm curl}_F \omega=0$ on $F$. Since $\curl \omega \times n_F=0$ on $F$ we have that $\curl \omega=0$ on $F$ or that $\curl \omega=0$ on $\p T$.

It then follows that $\omega_F = {\rm grad}_F p$ for some $p \in \pol_{r}(F)$.
Since $\omega$ vanishes on $\p F$, we may assume that $p$ and its derivatives vanishes
on $\p F$ as well.  That is, $\omega_F ={\rm grad}_F(b^2_F w)$ for some $w\in \pol_{r-6}(F)$,
where $b_F\in \pol_3(F)$ is the cubic face bubble corresponding to $F$.
It then follows from Stokes Theorem that, for all $\kappa \in {D}_{r-5}(F)$,
\begin{align*}
0 = \int_F \omega_F \cdot \kappa\, \mathrm{dA} = \int_f {\rm grad}_F (b^2_F w)\cdot \kappa\,\mathrm{dA}
= - \int_F b^2_F w ({\rm div}_F \kappa) \, \mathrm{dA}.
\end{align*}
Since ${\rm div}_F :{D}_{r-5}(F) \to \pol_{r-6}(F)$ is surjective, we conclude
that $w=0$, and so $\omega_F = 0$.  Therefore $\omega|_{\p T} =0$
and $\omega\in \mathring{M}_{d,r-1}(T^\s)$.  Finally, the DOFs \eqref{UDOFs8} implies that $\curl \omega = 0$ on $T$, and the 
DOFs \eqref{UDOFs7} then
give $\omega \equiv 0$.
\end{proof}

\begin{remark}\label{rem:H1curlRemark}
The proof of Lemma \ref{lem:H1curlDOFs}
shows that if $\omega\in M_{d,r-1}^1(T^\s)$ vanishes
on \eqref{UDOFs1}--\eqref{UDOFs6} restricted
to a single face $F\in \Delta_2(T)$, then $\omega|_F = \curl \omega |_F=0$.
\end{remark}

\begin{lemma}
\label{lem:H1DivergenceDOFs}
A function $\omega\in M_{d,r-2}^2(T^\s)$ ($r\ge 5$) is uniquely determined by
the values
\begin{subequations}
\label{eqn:H1DivergenceDOFs}
\begin{alignat}{4}
\label{eqn:H1DivergenceDOFs1}
&\omega(a),\ \Div \omega(a)\quad && \quad &&\forall a\in \Delta_0(T)\quad &&\text{($16$ DOFs)},\\
\label{eqn:H1DivergenceDOFs2}
&\int_e \omega \cdot \kappa\, \mathrm{ds}\quad &&\forall \kappa\in [\pol_{r-4}(e)]^3,\quad &&\forall e\in \Delta_1(T)\quad &&\text{($18(r-3)$ DOFs)},\\
\label{eqn:H1DivergenceDOFs3}
&\int_e (\Div \omega) \kappa\, \mathrm{ds}\quad &&\forall  \kappa\in \pol_{r-5}(e),\quad &&\forall e\in \Delta_1(T)\quad &&\text{($6(r-4)$ DOFs)},\\
\label{eqn:H1DivergenceDOFs4}
&\int_F \omega \cdot \kappa\, \mathrm{dA}\quad &&\forall \kappa \in [\pol_{r-5}(F)]^3,\quad &&\forall F\in \Delta_2(T)\quad &&\text{($6(r-3)(r-4)$ DOFs)},\\
\label{eqn:H1DivergenceDOFs5}
&\int_F (\Div \omega) \kappa\, \mathrm{dA}\quad &&\forall \kappa \in \pol_{r-6}(F),\quad &&\forall F\in \Delta_2(T)\quad &&\text{($2(r-4)(r-5)$ DOFs)},\\
\label{eqn:H1DivergenceDOFs6}
&\int_T \omega \cdot \kappa\, \mathrm{dx}\quad && \forall \kappa\in \curl \mathring{M}_{d,r-1}^1(T^\s), && &&\text{($\frac{(r-4)(r-3)(4r-11)}{3}$ DOFs)},\\
\label{eqn:H1DivergenceDOFs7}
&\int_T (\Div \omega)\kappa\, \mathrm{dx}\quad &&\forall \kappa\in \mathring{M}_{d,r-3}(T^\s),\quad && && \text{($\frac13 (2r-7)(r^2-7r+15)-1$ DOFs)}.
\end{alignat}
\end{subequations}
\end{lemma}
\begin{proof}
The number of degrees of freedom equals the dimension of $M_{d,r-2}^2(T^\s)$.
If $\omega$ vanishes at the DOFs, then standard arguments
show that $\omega=0$ and $\Div \omega=0$ on $\p T$ by using \eqref{eqn:H1DivergenceDOFs1}--\eqref{eqn:H1DivergenceDOFs5}.  Therefore $\Div \omega\in \mathring{M}_{d,r-3}^3(T^\s)$, and so \eqref{eqn:H1DivergenceDOFs7} implies that $\Div \omega=0$.
The exactness of the first sequence in \eqref{eqn:AllSequences} shows that $\omega = \curl \rho$ for some $\rho\in \mathring{M}_{d,r-1}^1(T^\s)$,
and therefore, using \eqref{eqn:H1DivergenceDOFs6}, we obtain $\omega\equiv 0$.
\end{proof}

\begin{lemma}
\label{thm:H1PressureDOFs}
Any $\omega\in M_{d,r-3}^3(T^\s)$ is uniquely determined by the degrees of freedom
\begin{subequations}
\label{eqn:H1PressureDOFs}
\begin{alignat}{4}
\label{eqn:H1PressureDOFs1}
&\omega(a)\quad && && \forall a\in \Delta_0(T)\quad &&\text{($4$ DOFs)},\\
\label{eqn:H1PressureDOFs2}
&\int_e \omega \kappa\, \mathrm{ds}\qquad &&\forall \kappa\in \pol_{r-5}(e),\quad &&\forall e\in \Delta_1(T)\quad && \text{($6(r-4)$ DOFs)},\\
\label{eqn:H1PressureDOFs3}
&\int_F \omega \kappa\, \mathrm{dA}\qquad &&\forall \kappa\in \pol_{r-6}(F),\quad &&\forall F\in \Delta_2(T)\quad && \text{($2(r-4)(r-5)$ DOFs)},\\
\label{eqn:H1PressureDOFs4}
&\int_T \omega \, \mathrm{dx} \qquad && && &&\text{($1$ DOFs)},\\
\label{eqn:H1PressureDOFs5}
&\int_T \omega \kappa  \, \mathrm{dx} \qquad &&\forall \kappa \in \mathring{M}_{d,r-3}^3(T^\s), && && \text{($\frac13 (2r-7)(r^2-7r+15)-1$ DOFs)}.
\end{alignat}
\end{subequations}
\end{lemma}
\begin{proof}
The boundary degrees of freedom \eqref{eqn:H1PressureDOFs}
are simply the Lagrange degrees of freedom, and so if $\omega\in M_{d,r-3}^3(T^\s)$ vanishes
on \eqref{eqn:H1PressureDOFs1}--\eqref{eqn:H1PressureDOFs3}, then we easily conclude
that $\omega|_{\p T}=0$.  If, in addition, $\omega$ vanishes on \eqref{eqn:H1PressureDOFs4}--\eqref{eqn:H1PressureDOFs5},
then we easily obtain that $\omega=0$, and that the degrees of freedom are unisolvent.
\end{proof}

With the degrees of freedom for $M_{d,r-k}^k(T^\s)$ established, 
we turn our attention to the continuous finite element spaces, $\Mrc(T^\s)$
and $M_{r-2}^2(T^\s)$.  The degrees of freedom of the former is given
in the next lemma.

\begin{lemma}\label{lem:VecHermiteDOFs}
Let $r\ge 5$.
A function 
$\omega\in \Mrc(T^\s)$ is uniquely determined by the values
\begin{subequations}
\label{eqn:VecHermiteDOFs}
\begin{alignat}{4}
\label{eqn:VecHermiteDOFs1}
&D^\alpha \omega(a)\qquad &&|\alpha|\le 1\quad &&\forall a\in \Delta_0(T)\qquad &&\text{($48$ DOFs)},\\
\label{eqn:VecHermiteDOFs2}
&\int_e \omega\cdot\kappa\,\mathrm{ds} \quad &&\forall \kappa\in [\pol_{r-5}(e)]^3 &&\forall e\in \Delta_1(T)\qquad &&
\hspace{-1cm}\text{($18(r-4)$ DOFs)},\\
\label{eqn:VecHermiteDOFs3}
&
\int_e (\curl \omega|_F \cdot n_F)\kappa\, \mathrm{ds}\quad &&\forall \kappa\in \pol_{r-4}(e)&& \forall e\in \Delta_1(F),\ \forall F\in \Delta_2(T)\ \  &&\text{($12(r-3)$ DOFs)},\\
\label{eqn:VecHermiteDOFs4}
&\int_F (\omega\cdot n_F)\kappa\, \mathrm{dA}\quad &&\forall \kappa\in \pol_{r-4}(F) &&\forall F\in \Delta_2(T) && 
\hspace{-1.25cm}\text{($2(r-2)(r-3)$ DOFs)},\\
\label{eqn:VecHermiteDOFs6}
&\int_F (n_F\times \omega \times n_F)\cdot \kappa \, \mathrm{dA}\quad &&\forall \kappa\in D_{r-5}(F) &&\forall F\in \Delta_2(T) && 
\hspace{-1.25cm}\text{($4(r-3)(r-5)$ DOFs)},\\
\label{eqn:VecHermiteDOFs7}
&\int_K \omega \cdot \kappa\, \mathrm{dx} \quad &&\forall \kappa\in \grad \mathring{M}_{d,r}^0(T^\s)
&& &&\hspace{-2.25cm} \text{($\frac{2(r-4)(r-3)(r-2)}{3}$ DOFs)},\\
\label{eqn:VecHermiteDOFs8}
&\int_K \curl \omega \cdot \kappa\, \mathrm{dx} \qquad &&\forall \kappa\in \curl \Mrco(T^\s)
&& && \hspace{-2.25cm} \text{($\frac{(r-3)(4r^2+3r+14)}3$ DOFs)},
\end{alignat}
\end{subequations}
where we recall that $D_{r-5}(F)$ is the local Raviart--Thomas space on the face $F$.
\end{lemma}
\begin{proof}
The total number of conditions in \eqref{eqn:VecHermiteDOFs} 
is $(r-2)(2r^2+r+9)$ which equals the dimension of $\Mrc(T^\s)$.
Suppose that $\omega\in \Mrc(T^\s)$ vanishes on the DOFs.
We show that $\omega \equiv 0$.  This is done by
adopting similar arguments as the proof of Lemma \ref{lem:H1curlDOFs}. 

In this case, it is easy to see that $\omega|_e=0$ on all $e\in \Delta_1(T)$, 
$(\curl \omega|_F)\cdot n_F|_e=0$ on all $e\in \Delta_1(F)$ and $F\in \Delta_2(T)$,
and $\omega\cdot n_F|_F=0$ on all $F\in \Delta_2(T)$.

Applying Stokes Theorem, we find that, for any $q\in \pol_{r-5}(T)$,
\begin{align*}
\int_F ({\rm curl}_F \omega)q\, \mathrm{dA} 
& = \int_F ({\rm rot}_F q)\cdot \omega\, \mathrm{dA}\\
& = \int_F \grad \cdot (n_F \times \omega)\, \mathrm{dA}\\
& = \int_F (\grad q\times n_F)\cdot (n_F \times \omega \times n_F)\, \mathrm{dA}=0,
\end{align*}
where we have used \eqref{eqn:VecHermiteDOFs6} in the last equality.
Because ${\rm curl}_F \omega = (\curl \omega)\cdot n_F$,
and $(\curl \omega \cdot n_F)$ vanishes on the edges of $F$,
we conclude that ${\rm curl}_F \omega=0$ on $F$.

By using the same arguments
as in Lemma \ref{lem:H1curlDOFs}, we conclude
that $\omega_F = {\rm grad}_F (b_F^2 w)$ for some $w\in \pol_{r-6}(F)$.
Consequently, we have
\begin{align*}
0 = \int_F \omega_F \cdot \kappa\, \mathrm{dA} = -\int_F b_F^2 w ({\rm div}_F \kappa)\, \mathrm{dA}\quad \forall \kappa\in D_{r-5}(F),
\end{align*}
and therefore $w=0$ and $\omega_F|_F=0$.  Thus, $\omega=0$ on $F$.
Finally, it follows from \eqref{eqn:VecHermiteDOFs8} that
$\curl \omega=0$, and therefore, by \eqref{eqn:VecHermiteDOFs7}, $\omega=0$.
\end{proof}

\begin{lemma}\label{lem:H1divDOFs}
Let $r\ge 5$, then
a function $\omega \in M_{r-2}^2 (T^\s)$ is uniquely determined by
the following degrees of freedom:
\begin{subequations}
\label{XDOFs}
\begin{alignat}{4}
\label{XDOFs1}
&\omega(a)\qquad &&
&& \forall a\in \Delta_0(T)\quad &&(\text{$12$ DOFs}),\\
\label{XDOFs2}
&\int_e \omega \cdot \kappa\,\mathrm{ds},\qquad &&\forall \kappa\in [\pol_{r-4}(e)]^3,\quad &&\forall e\in\Delta_1(T)\qquad &&
\text{($18(r-3)$ DOFs)},
\\
\label{XDOFs3}
&\int_f \omega \cdot \kappa\, \mathrm{dA} \quad &&\forall \kappa\in [\pol_{r-5}(F)]^3,\quad &&\forall F\in \Delta_2(T)\qquad &&\text{($6(r-4)(r-3)$ DOFs)}\\
\label{XDOFs4}
&\int_T \omega \cdot \kappa \, \mathrm{dx} \quad &&\forall \kappa\in 
\curl \mathring{M}_{d,r-1}^1(T^\s), && &&
\!\!\!\!\!\!\!\!\text{($\frac{(r-4)(r-3)(4r-11)}{3}$ DOFs)},\\
\label{XDOFs5}
&\int_T (\dive \omega)\, \kappa \, \mathrm{dx} \quad &&\forall \kappa\in \mathring{V}_{d,r-3}(T^\s), && &&
\!\!\!\!
\!\text{($\frac{2(r-2)(r-1)r}{3}-1$ DOFs)}.
\end{alignat}
\end{subequations}
\end{lemma}
\begin{proof}
 The number of conditions is $2r^3-9r^2+19r-15$ which equals the dimension of $M_{r-2}^2(T^\s)$.
We show that if $\omega \in M_{r-2}^2(T^\s)$ vanishes
at \eqref{XDOFs}, then $\omega \equiv 0$.

In this case, it is easy to see that $\omega= 0$ on the boundary $\partial T$ by the DOFs \eqref{XDOFs1}--\eqref{XDOFs3}.
Hence, $\omega \in \mathring{M}_{r-2}^{2}(T^\s)$. 
Then, the DOFs \eqref{XDOFs5} implies that $\dive \omega =0$. Finally, the exactness of the sequence \eqref{eqn:AllSequencesC2} and 
the DOFs \eqref{XDOFs4} implies that $\omega \equiv 0$.
\end{proof}

A unisolvent set of degrees of freedom 
of the spaces $\Vrc(T^\s)$
is given in the following lemma.  The result
essentially follows from \cite{Stenberg10}.
\begin{lemma}\label{lem:StenbergHDiv}
Let $r\ge 5$, then a function $\omega \in \Vrc(T^\s)$ is uniquely determined 
by the values
\begin{subequations}
\label{eqn:StenbergHDiv}
\begin{alignat}{4}
\label{eqn:StenbergHDiv1}
&\omega(a)\quad && && \forall a\in \Delta_0(T)\quad &&\text{($12$ DOFs)},\\
\label{eqn:StenbergHDiv2}
&\int_e (\omega\cdot n_F) \kappa\, \mathrm{ds}\ &&\forall \kappa \in \pol_{r-4}(e)\ \ &&\forall e\in \Delta_1(F),\ \forall F\in \Delta_2(T)\ &&\text{($12(r-3)$ DOFs),}\\
\label{eqn:StenbergHDiv3}
&\int_F (\omega\cdot n_F)\kappa\, \mathrm{dA} \quad &&\forall \kappa\in {\pol}_{r-5}(F) 
\quad &&\forall f\in \Delta_2(T)\qquad &&\text{($2(r-3)(r-4)$ DOFs)},\\
\label{eqn:StenbergHDiv4}
&\int_T \omega \cdot \kappa\, \mathrm{dx} \quad &&\forall \kappa\in \curl \Mrco(T^\s)
&& && \hspace{-0.5cm}  \text{($\frac{(r-3)(4r^2+3r+14)}{3}$ DOFs)},\\
\label{eqn:StenbergHDiv5}
&\int_T (\Div\omega) \kappa \, \mathrm{dx}\quad &&\forall \kappa\in \mathring{V}_{d,r-3}^3(T^\s) 
&& && \hspace{-0.5cm}  \text{($\frac{2 (r-2)(r-1)r}3 -1$ DOFs)}.
\end{alignat}
\end{subequations}
\end{lemma}

\begin{lemma}\label{lem:L2DOFs}
Let $r\ge 5$, then
a function $\omega \in V_{d,r-3}^3 (T^\s)$ is uniquely determined by
the following degrees of freedom:
\begin{subequations}
\label{WDOFs}
\begin{alignat}{4}
\label{WDOFs1}
&\int_T \omega \, \mathrm{dx}, && &&\quad\quad \text{($1$ DOFs)},\\
\label{WDOFs2}
&\int_T \omega \, \kappa \, \mathrm{dx} \quad &&\forall \kappa\in \mathring{V}_{d,r-3}^3(T^\s), &&\quad\quad
\text{($\frac{2(r-2)(r-1)r}{3}-1$ DOFs)}.
\end{alignat}
\end{subequations}
\end{lemma}
\begin{proof}
 Trivial.
\end{proof}

\subsection{Commuting Projections}
In this section we show that
the degrees of freedom given in the previous section
yield  projections that commute with the differential operators.
We first consider the sequence with highest smoothness.

\begin{theorem}\label{thm:LocalProjectionTwo}
 Let $\Pi_{d,0}: C^\infty(T)\rightarrow M_{d,r}^0(T^\s)$ be the projection induced by the DOFs \eqref{eqn:SigmaDOFs}, that is, 
 \[
  \phi(\Pi_{d,0} p) = \phi(p), \quad \forall \phi \in \sf{DOFs} \text{ in \eqref{eqn:SigmaDOFs}}.
 \]
  Likewise, let $\Pi_{d,1}:[ C^\infty(T)]^3\rightarrow M_{d,r-1}^1(T^\s)$ be the   projection induced by the DOFs \eqref{UDOFs},
  $\Pi_{d,2}:[C^\infty(T)]^3 \rightarrow M_{d,r-2}^2(T^\s)$ be the   projection induced by the DOFs \eqref{eqn:H1DivergenceDOFs},
  and $\Pi_{d,3}: C^\infty(T) \rightarrow M^3_{d,r-3}(T^\s)$ be the  projection induced by the DOFs \eqref{eqn:H1PressureDOFs}.
  Then for $r\ge 5$ the following diagram commutes
  \begin{alignat*}{2}
  \begin{array}{ccccccccccccc}
&\mathbb{R}& {\xrightarrow{\hspace*{0.1cm}}}&
 C^\infty(T)\!\!\!
&\stackrel{\grad}{\xrightarrow{\hspace*{0.5cm}}}&
\!\!\![C^\infty(T)]^3
\!\!\!&\stackrel{\curl}{\xrightarrow{\hspace*{0.5cm}}}&
\!\!\!  [C^\infty(T)]^3
\!\!\!\!& \stackrel{{\rm div}}{\xrightarrow{\hspace*{0.5cm}}}& 
\!\!\! C^\infty(T)
&
\!\!\!\xrightarrow{\hspace*{0.1cm}}  &0\\
&&&\;\;\;\;\;\;\;\;\;\;\;\;\downarrow {\scriptstyle \Pi_{d,0}} & &
\downarrow{\scriptstyle \Pi_{d,1}}
& &\downarrow{\scriptstyle  \Pi_{d,2}}
& &
\;\;\;\;\downarrow{\scriptstyle \Pi_{d,3}}\\
&\mathbb{R}& {\xrightarrow{\hspace*{0.1cm}}}&
M_{d,r}^0(T^\s)\!\!\!
&\stackrel{\grad}{\xrightarrow{\hspace*{0.5cm}}}&
M_{d,r-1}^1(T^\s) 
\!\!\!&\stackrel{\curl}{\xrightarrow{\hspace*{0.5cm}}}&
\!\!\!M_{d,r-2}^2(T^\s)
\!\!\!\!& \stackrel{{\rm div}}{\xrightarrow{\hspace*{0.5cm}}}& 
\!\!\!M^3_{d,r-3}(T^\s)
&
\!\!\!\xrightarrow{\hspace*{0.1cm}}  &0.
\end{array}
\end{alignat*}
More specifically, we have 
\begin{subequations}
\label{ABCcommute}
\begin{alignat}{2}
\label{ABCcommute-1}
\grad\, \Pi_{d,0} p = &\; \Pi_{d,1}\,\grad p,&& \quad \forall p \in C^\infty(T)\\
\label{ABCcommute-2}
\curl\, \Pi_{d,1} p = &\; \Pi_{d,2}\,\curl  p,&&\quad 
\forall p\in  [C^\infty(T)]^3,\\
\label{ABCcommute-3}
\dive \, \Pi_{d,2}  p =  &\;
\Pi_{d,3}\,\dive p,&&\quad 
\forall p\in \bld [C^\infty(T)]^3.
\end{alignat}
\end{subequations}
\end{theorem}
\begin{proof}

(i) {\em Proof of \eqref{ABCcommute-1}.}
We take $p\in C^\infty(T)$. 
Since $\rho:=\grad\, \Pi_0 p- \Pi_1\,\grad p \in M_{d,r-1}^1(T^\s)$, we only need to prove that 
$\rho$ vanishes at the DOFs \eqref{UDOFs}. 

For the vertex-based terms, 
we have, for all $|\alpha|\le 1$ and $a\in \Delta_0(T)$, 
\begin{align*}
 D^\alpha \rho(a) = 
 D^{\alpha}(\grad\, \Pi_{d,0} p(a)
 - \Pi_{d,1} \grad p(a)) = 
 0,
\end{align*}
by the definition of $\Pi_{d,0}$, $\Pi_{d,1}$ and the DOFs \eqref{eqn:C1DOF1}, \eqref{UDOFs1}.

For the edge-based terms, we have, for all $\kappa\in [\pol_{r-5}(e)]^3$,
\begin{align*}
\int_e \rho\cdot\kappa\,\mathrm{ds}
= & \;
\int_e(\grad\, \Pi_{d,0} p-\grad p)\cdot\kappa\,\mathrm{ds} \quad & \text{by} \eqref{UDOFs2} \\
= &\;
\int_e\sum_{i\in\{+,-\}}
\frac{\partial(\Pi_{d,0} p- p)}{\partial n_e^{i}}(\kappa\cdot n_e^{i})\,\mathrm{ds}
+
\int_e
\frac{\partial(\Pi_{d,0} p- p)}{\partial t_e}(\kappa\cdot t_e)\,\mathrm{ds}\\
= &
\int_e \frac{\partial(\Pi_{d,0} p- p)}{\partial t_e}(\kappa\cdot t_e)\,\mathrm{ds} \quad & \text{ by } \eqref{eqn:C1DOF22} \\
=& -\int_e
(\Pi_0
 p- p)\frac{\partial(\kappa\cdot t_e)}{\partial t_e}\,\mathrm{ds}\quad & \text{ by } \eqref{eqn:C1DOF1} \\
 =&0, \quad & \text{by} \eqref{eqn:C1DOF21}
\end{align*}

We also have by the definition of $\Pi_{d,1}$ and \eqref{UDOFs3}
\begin{align*}
\int_e(\curl \rho)\cdot\kappa\,\mathrm{ds}
= &\;
\int_e(\curl(\grad\, (\Pi_{d,0} p-p))\cdot\kappa\,\mathrm{ds}=0 \quad \forall \kappa\in [\pol_{r-4}(e)]^3
\end{align*}

For the face-based terms, we have, for all $\kappa\in \pol_{r-4}(F)$,
\begin{align*}
 \int_F(\rho\cdot n_F)\kappa \,\mathrm{dA}
 = 
  \int_F(\grad\Pi_{d,0} p- \Pi_{d,1} \grad p)\cdot n_F\kappa \,\mathrm{dA} = 0
\end{align*}
by the the definitions of $\Pi_{d,0}$, $\Pi_{d,1}$  and the DOFs \eqref{eqn:C1DOF32}, \eqref{UDOFs4}.
We also have, for all $\kappa\in {D}_{r-5}(F)$, using the definition of $\Pi_{d,1}$ and \eqref{UDOFs5},
\begin{align*}
\int_F ({n_F\times }\rho \times n_F) \cdot \kappa\, \mathrm{dA}
=&\int_F ({n_F\times } \grad (\Pi_{d,0} p-p)  \times n_F) \cdot \kappa\, \mathrm{dA} \\
=&\; \int_F {\rm grad}_F(\Pi_{d,0} p- p) \cdot \kappa\, \mathrm{dA}\\
=&\;
-\int_F (\Pi_{d,0} p-  p) {\rm div}_{F} \kappa\, \mathrm{dA}
+
\int_{\partial F} (\Pi_{d,0} p-  p) \kappa \cdot n_{\p F} \, \mathrm{ds},
\end{align*}
where $n_{\p F}$ is unit normals tangent to $F$ and perpendicular to the edges of $F$. 
Since for $\kappa \in  {D}_{r-5}(F)$, $ {\rm div}_F \kappa \in \pol_{r-6}(F)$, and 
$\kappa \cdot n_{\p F}|_e \in \pol_{r-6}(e)$ for all three edges $e$ of $F$, the right hand side of the 
above expression vanishes by the DOFs \eqref{eqn:C1DOF21} and \eqref{eqn:C1DOF31}.
Moreover, we have by the definition of $\Pi_{d,1}$ and \eqref{UDOFs6}, 
\begin{align*}
\int_F ({\curl \rho \times n_F}) \cdot \kappa\, \mathrm{dA}
=&\;
\int_F ({\curl \grad(\Pi_{d,0} p- p
) \times n_F}) \cdot \kappa\, \mathrm{dA}
=0 \quad \forall  \kappa\in [\pol_{r-5}(F)]^3.
\end{align*}

For the cell-based terms, we have,  for $\kappa\in \grad \mathring{M}_{d,r}^{0}(T^\s)$, 
\begin{align*}
\int_T \rho \cdot \kappa\, \mathrm{dx}
=&\;
\int_T (\grad \Pi_{d,0} p- \Pi_1 \grad p
) \cdot \kappa\, \mathrm{dx}
=0
\end{align*}
using the definitions of $\Pi_{d,0}, \Pi_{d,1}$ and the DOFs \eqref{eqn:C1DOF4} and  \eqref{UDOFs7}. 
We also have using  definition of $\Pi_{d,1}$ and \eqref{UDOFs8},
\begin{align*}
\int_T \curl \rho \cdot \kappa\, \mathrm{dx}
=&\;
\int_T (\curl\grad(\Pi_{d,0} p- p
)) \cdot \kappa\, \mathrm{dx}
=0.
\end{align*}
Combining the above results, we conclude that $\rho = \grad\, \Pi_{d,0} p-  \Pi_{d,1}\,\grad p = 0$. This completes the proof for the 
identity \eqref{ABCcommute-1}.

(ii) {\em Proof of \eqref{ABCcommute-2}}: Let $p\in [C^\infty(T)]^3$ and set $\rho = \curl \Pi_{d,1} p - \Pi_{d,2} \curl p\in M_{d,r-2}^2(T^\s)$.  
 We show that $\rho$ vanishes at the DOFs \eqref{eqn:H1DivergenceDOFs}.
 
 First, we have for all $a\in \Delta_0(T)$,
 \[
 \rho(a) = (\curl \Pi_{d,1} p) (a) - (\Pi_{d,2} \curl p)(a) = 0
 \]
 by \eqref{UDOFs1} and \eqref{eqn:H1DivergenceDOFs1}.  Furthermore, we have
 \[
 \Div \rho (a) = -\Div \Pi_{d,2} \curl p(a) = - \Div \curl p(a) = 0
 \]
 by \eqref{eqn:H1DivergenceDOFs1}.  Similar arguments show that
\begin{alignat*}{4}
&\int_e (\Div \rho) \kappa\, \mathrm{ds}=0\qquad &&\forall \kappa\in \pol_{r-5}(e),\quad &&\forall e\in \Delta_1(T),\\
&\int_F (\Div \rho)\kappa\, \mathrm{dA}=0\qquad &&\forall \kappa\in \pol_{r-6}(F),\quad &&\forall F\in \Delta_2(T),\\
&\int_T (\Div \rho)\kappa\, \mathrm{dx}=0\qquad && && \forall \kappa\in \mathring{M}_{d,r-3}(T^\s)
\end{alignat*} 
by using \eqref{eqn:H1DivergenceDOFs3}, \eqref{eqn:H1DivergenceDOFs5}, and \eqref{eqn:H1DivergenceDOFs7}.
 
Next we have, for $\kappa\in [\pol_{r-4}(e)]^3$,
\begin{align*}
\int_e \rho\cdot \kappa\, \mathrm{ds} = \int_e (\curl \Pi_{d,1} p - \curl p)\cdot \kappa\, \mathrm{ds} =0
\end{align*}
by \eqref{UDOFs3} and \eqref{eqn:H1DivergenceDOFs2}.  

Let $F\in \Delta_2(T)$ and $\kappa\in [\pol_{r-5}(F)]^3$.  We then have
\begin{align*}
\int_F (\rho \times n_F) \cdot \kappa\, \mathrm{dA} = \int_F (\curl \Pi_{d,1} p - \curl p)\times n_F \cdot  \kappa\, \mathrm{dA}=0
\end{align*}
by \eqref{UDOFs6} and \eqref{eqn:H1DivergenceDOFs4}.  We also have, for $\kappa\in \pol_{r-5}(F)$,
\begin{align*}
\int_F (\rho \cdot n_F)\kappa\, \mathrm{dA} 
&= \int_F (\curl \Pi_{d,1} p - \curl p)\cdot n_F \kappa\, \mathrm{dA}\\
&= \int_F ({\rm curl}_F \Pi_{d,1} p - {\rm curl}_F p)\kappa\, \mathrm{dA}\\
& = \int_F ({\rm rot}_F \kappa \cdot (\Pi_{d,1} p - p)\, \mathrm{dA} + \int_{\p F} (\Pi_{d,1} p - p)\cdot t_{\p F} \kappa\, \mathrm{ds}=0
\end{align*}
by \eqref{eqn:H1DivergenceDOFs4}, \eqref{UDOFs5}, and \eqref{UDOFs2}.

Finally, we have for all $\kappa\in \curl \mathring{M}_{d,r-1}^1(T^\s)$,
\begin{align*}
\int_T \rho\cdot \kappa\, \mathrm{dx} = \int_T (\curl \Pi_{d,1} p - \curl p)\cdot \kappa\, \mathrm{dx}=0
\end{align*}
by \eqref{UDOFs8} and \eqref{eqn:H1DivergenceDOFs6}.  Thus, $\rho$ vanishes
on \eqref{eqn:H1DivergenceDOFs}, and so $\rho\equiv 0$.

(iii) {\em Proof of \eqref{ABCcommute-3}}: Let $p\in [C^\infty(T)]^3$ and set $\rho = \Div \Pi_{d,2} p - \Pi_{d,3} \Div p\in M_{d,r-3}^3(T^\s)$.  
Similar to parts (i)--(ii), we show that $\rho$ vanishes on \eqref{eqn:H1PressureDOFs}.

First, we clearly have $\rho(a)=0$ by \eqref{eqn:H1PressureDOFs1} and \eqref{eqn:H1DivergenceDOFs1}.
We further have
\begin{align*}
\int_e \rho \kappa\, \mathrm{ds}=0\qquad \forall \kappa\in \pol_{r-5}(e),\quad \forall e\in \Delta_1(T)
\end{align*}
by \eqref{eqn:H1PressureDOFs2} and \eqref{eqn:H1DivergenceDOFs3};
\begin{align*}
\int_F \rho \kappa\, \mathrm{dA}=0\qquad \forall \kappa\in \pol_{r-6}(F),\quad \forall F\in \Delta_2(T)
\end{align*}
by \eqref{eqn:H1PressureDOFs3} and \eqref{eqn:H1DivergenceDOFs5}; and
\begin{align*}
\int_T \rho \kappa\, \mathrm{dx}=0\qquad \forall \kappa\in \mathring{M}_{d,r-3}^3(T^\s)
\end{align*}
by \eqref{eqn:H1PressureDOFs5} and \eqref{eqn:H1DivergenceDOFs7}.  Finally, we have
\begin{align*}
\int_T \rho\, \mathrm{dx} 
&= \int_T \big(\Div \Pi_{d,2} p - \Div p\big)\, \mathrm{dx}\\
& = \int_{\p T} (\Pi_{d,2} p - p)\cdot n\, \mathrm{dA}=0
\end{align*}
by \eqref{eqn:H1PressureDOFs4} and \eqref{eqn:H1DivergenceDOFs4}.
Thus, $\rho$ vanishes on \eqref{eqn:H1PressureDOFs}, and so $\rho\equiv 0$.
\end{proof}

By using similar arguments
as those given in Theorem \ref{thm:LocalProjectionTwo},
we obtain commutative projections
for the second sequence in \eqref{eqn:AllSequencesC2}.
The proof is given in the appendix.
\begin{theorem}\label{thm:LocalProjectionOne}
 Let $\Pi_{d,0}: C^\infty(T)\rightarrow M_{d,r}^0(T^\s)$ be the projection induced by the DOFs \eqref{eqn:SigmaDOFs}, 
  $\Pi_{d,1}:[ C^\infty(T)]^3\rightarrow M_{d,r-1}^1(T^\s)$ be    projection induced by the DOFs \eqref{UDOFs},
  $\Pi_{2}:[C^\infty(T)]^3 \rightarrow M_{r-2}^2(T^\s)$ be the   projection induced by the DOFs \eqref{XDOFs},
  and $\Pi_{3}: C^\infty(T) \rightarrow V^3_{d,r-3}(T^\s)$ be the   projection induced by the DOFs \eqref{WDOFs}.
  Then for $r\ge 5$ the following diagram commutes
  \begin{alignat*}{2}
  \begin{array}{ccccccccccccc}
&\mathbb{R}& {\xrightarrow{\hspace*{0.1cm}}}&
 C^\infty(T)\!\!\!
&\stackrel{\grad}{\xrightarrow{\hspace*{0.5cm}}}&
\!\!\![C^\infty(T)]^3
\!\!\!&\stackrel{\curl}{\xrightarrow{\hspace*{0.5cm}}}&
\!\!\!  [C^\infty(T)]^3
\!\!\!\!& \stackrel{{\rm div}}{\xrightarrow{\hspace*{0.5cm}}}& 
\!\!\! C^\infty(T)
&
\!\!\!\xrightarrow{\hspace*{0.1cm}}  &0\\
&&&\;\;\;\;\;\;\;\;\;\;\;\;\downarrow {\scriptstyle \Pi_{d,0}} & &
\downarrow{\scriptstyle \Pi_{d,1}}
& &\downarrow{\scriptstyle  \Pi_2}
& &
\;\;\;\;\downarrow{\scriptstyle \Pi_3}\\
&\mathbb{R}& {\xrightarrow{\hspace*{0.1cm}}}&
M_{d,r}^0(T^\s)\!\!\!
&\stackrel{\grad}{\xrightarrow{\hspace*{0.5cm}}}&
M_{d,r-1}^1(T^\s) 
\!\!\!&\stackrel{\curl}{\xrightarrow{\hspace*{0.5cm}}}&
\!\!\!M_{r-2}^2(T^\s)
\!\!\!\!& \stackrel{{\rm div}}{\xrightarrow{\hspace*{0.5cm}}}& 
\!\!\!V^3_{d,r-3}(T^\s)
&
\!\!\!\xrightarrow{\hspace*{0.1cm}}  &0.
\end{array}
\end{alignat*}
More specifically, we have 
\begin{subequations}
\label{commute}
\begin{alignat}{2}
\label{commute-1}
\grad\, \Pi_{d,0} p = &\; \Pi_{d,1}\,\grad p,&& \quad \forall p \in C^\infty(T)\\
\label{commute-2}
\curl\, \Pi_{d,1} p = &\; \Pi_2\,\curl  p,&&\quad 
\forall p\in  [C^\infty(T)]^3,\\
\label{commute-3}
\dive \, \Pi_2  p =  &\;
\Pi_3\,\dive p,&&\quad 
\forall p\in \bld [C^\infty(T)]^3.
\end{alignat}
\end{subequations}
\end{theorem}

Finally, we state the commutative projections
for the third sequence in \eqref{eqn:AllSequencesC2}.
The proof is given in the appendix. 
\begin{theorem}\label{thm:LocalProjectionThree}
 Let $\Pi_{d,0}: C^\infty(T)\rightarrow M_{d,r}^0(T^\s)$ be the projection induced by the DOFs \eqref{eqn:SigmaDOFs},
  $\Pi_{c,1}:[ C^\infty(T)]^3\rightarrow \Mrc(T^\s)$ be the   projection induced by the DOFs \eqref{eqn:VecHermiteDOFs},
  $\Pi_{c,2}:[C^\infty(T)]^3 \rightarrow \Vrc(T^\s)$ be the   projection induced by the DOFs \eqref{eqn:StenbergHDiv},
  and $\Pi_{3}: C^\infty(T) \rightarrow V^3_{d,r-3}(T^\s)$ be the   projection induced by the DOFs \eqref{WDOFs}.
  Then, 
  the following diagram commutes
  \begin{alignat*}{2}
  \begin{array}{ccccccccccccc}
&\mathbb{R}& {\xrightarrow{\hspace*{0.1cm}}}&
 C^\infty(T)\!\!\!
&\stackrel{\grad}{\xrightarrow{\hspace*{0.5cm}}}&
\!\!\![C^\infty(T)]^3
\!\!\!&\stackrel{\curl}{\xrightarrow{\hspace*{0.5cm}}}&
\!\!\!  [C^\infty(T)]^3
\!\!\!\!& \stackrel{{\rm div}}{\xrightarrow{\hspace*{0.5cm}}}& 
\!\!\! C^\infty(T)
&
\!\!\!\xrightarrow{\hspace*{0.1cm}}  &0\\
&&&\;\;\;\;\;\;\;\;\;\;\;\;\downarrow {\scriptstyle \Pi_{d,0}} & &
\downarrow{\scriptstyle \Pi_{c,1}}
& &\downarrow{\scriptstyle  \Pi_{c,2}}
& &
\;\;\;\;\downarrow{\scriptstyle \Pi_3}\\
&\mathbb{R}& {\xrightarrow{\hspace*{0.1cm}}}&
M_{d,r}^0(T^\s)\!\!\!
&\stackrel{\grad}{\xrightarrow{\hspace*{0.5cm}}}&
\Mrc(T^\s) 
\!\!\!&\stackrel{\curl}{\xrightarrow{\hspace*{0.5cm}}}&
\!\!\!\Vrc(T^\s)
\!\!\!\!& \stackrel{{\rm div}}{\xrightarrow{\hspace*{0.5cm}}}& 
\!\!\!V^3_{d,r-3}(T^\s)
&
\!\!\!\xrightarrow{\hspace*{0.1cm}}  &0.
\end{array}
\end{alignat*}
More specifically, we have 
\begin{subequations}
\label{commuteABD}
\begin{alignat}{2}
\label{DEFcommute-1}
\grad\, \Pi_{d,0} p = &\; \Pi_{c,1}\,\grad p,&& \quad \forall p \in C^\infty(T)\\
\label{DEFcommute-2}
\curl\, \Pi_{c,1} p = &\; \Pi_{c,2}\,\curl  p,&&\quad 
\forall p\in  [C^\infty(T)]^3,\\
\label{DEFcommute-3}
\dive \, \Pi_{c,2}  p =  &\;
\Pi_3\,\dive p,&&\quad 
\forall p\in \bld [C^\infty(T)]^3.
\end{alignat}
\end{subequations}
\end{theorem}

\section{Global Smooth Finite Element de Rham Complexes in Three Dimensions}\label{section-Global}
In this section we study
the global finite element spaces
induced by the degrees of freedom
in Section \ref{subsec-DOFs}.
To this end, we suppose that $\Omega \subset \mathbb{R}^3$
is a polyhedral domain.  Let $\mct$ be a simplicial triangulation
of $\Omega$, and let $\mct^\s$ be the simplicial triangulation obtained
by adding connecting each barycenter of $T\in \mct$ with its vertices, i.e.,
$\mct^\s$ is obtained by performing an Alfeld split to each $T\in \mct$.

The degrees of freedom in Lemmas \ref{lem:C1-DOF}, \ref{lem:H1curlDOFs}, \ref{lem:H1DivergenceDOFs}--\ref{lem:L2DOFs},
(cf.~Remarks \ref{rem:H2Remark} and \ref{rem:H1curlRemark})
lead to the following global spaces, for $r\ge 5$,
\begin{align*}
M_{d,r}^0(\mct^\s) &= \{\omega \in C^1(\Omega): \omega |_T\in M_{d,r}^0(T^\s)\ \forall T\in \mct, \text{ $\omega$ is $C^2$ at vertices}\},\\
M_{d,r-1}^1(\mct^\s) &= \{\omega \in [C^0(\Omega)]^3:\curl \omega \in [C^0(\Omega)]^3,\   \omega|_T\in M_{d,r-1}^1(T^\s) \ \forall T\in \mct,\ \text{$\omega$ is $C^1$ at vertices}\},\\
M_{d,r-2}^2(\mct^\s) & =\{\omega\in [C^0(\Omega)]^3:\ \Div \omega\in C^0(\Omega),\ \omega|_T\in M_{d,r-2}^2(T^\s)\ 
\forall T\in \mct\},\\
M_{d,r-3}^3(\mct^\s) & = \{\omega\in C^0(\Omega):\ \omega|_T\in M_{d,r-3}^3(T^\s)\ \forall T\in \mct\},\\
\Mrc(\mct^\s)
& = \{\omega\in [C^0(\Omega)]^3:\ \omega|_T\in \Mrc(T^\s)\ \forall T\in \mct, \text{ $\omega$ is $C^1$
at vertices}\},\\
M_{r-2}^2(\mct^\s)& = \{\omega \in  [C^0(\Omega)]^3 : \omega|_T \in M_{r-2}^2(T^\s
)\ \forall T\in \mct\},\\
\Vrc(\mct^\s)
& = \{\omega\in L^2(\Omega):\ \Div \omega\in L^2(\Omega),\ \omega|_T\in \Vrc(T^\s)\ \forall T\in \mct,
\text{ $\omega$ is $C^0$ at vertices}\},\\
V_{d,r-3}^3(\mct^\s) & = \{\omega \in L^2(\Omega): \omega|_T\in V^3_{d,r-3}(T^\s)\ \forall T\in \mct\}.
\end{align*}
Clearly the following sequences of spaces
%
\begin{subequations}
\label{eqn:Globalcomplex3D2}
\begin{alignat}{4}
\label{eqn:Globalcomplex3D21}
&\mathbb{R}\
 {\xrightarrow{\hspace*{0.5cm}}}\
M_{d,r}^0(\mct^\s)\
\stackrel{\grad}{\xrightarrow{\hspace*{0.5cm}}}\
  M_{d,r-1}^1(\mct^\s)\
&&\stackrel{\curl}{\xrightarrow{\hspace*{0.5cm}}}\
M_{d,r-2}^2(\mct^\s)\
&&\stackrel{{\dive}}{\xrightarrow{\hspace*{0.5cm}}}\
M_{d,r-3}^3(\mct^\s)\
&&\xrightarrow{\hspace*{0.5cm}}\  
0,\\
\label{eqn:Globalcomplex3D22}
&\mathbb{R}\
 {\xrightarrow{\hspace*{0.5cm}}}\
M_{d,r}^0(\mct^\s)\
\stackrel{\grad}{\xrightarrow{\hspace*{0.5cm}}}\
M_{d,r-1}^1(\mct^\s)\
&&\stackrel{\curl}{\xrightarrow{\hspace*{0.5cm}}}\
M_{r-2}^2(\mct^\s)\
&&\stackrel{{\dive}}{\xrightarrow{\hspace*{0.5cm}}}\
V_{d,r-3}^3(\mct^\s)\
&&\xrightarrow{\hspace*{0.5cm}}\  0,\\
\label{eqn:Globalcomplex3D23}
&\mathbb{R}\
{\xrightarrow{\hspace*{0.5cm}}}\
M_{d,r}^0(\mct^\s)\
\stackrel{\grad}{\xrightarrow{\hspace*{0.5cm}}}\
 \Mrc(\mct^\s)\
&&\stackrel{\curl}{\xrightarrow{\hspace*{0.5cm}}}\
\Vrc(\mct^\s)\
&&\stackrel{{\dive}}{\xrightarrow{\hspace*{0.5cm}}}\
V_{d,r-3}^3(\mct^\s)\
&&\xrightarrow{\hspace*{0.5cm}}\  0
\end{alignat}
\end{subequations}
forms a complex.  In addition, we can define commuting projections. For example, for the first sequence we can define 
$\pi_{d,i}$ such that $\pi_{d,i} \omega |_T= \Pi_{d,i} (\omega|_T)$ for all $T \in \mct$, and 
by using Theorem \ref{thm:LocalProjectionTwo}, we get the following commuting diagram for the second sequence
%
%
\eqref{eqn:Globalcomplex3D22}:
\begin{alignat*}{2}
\begin{array}{ccccccccccccc}
&\mathbb{R}& {\xrightarrow{\hspace*{0.1cm}}}&
 C^\infty(\Omega)\!\!\!
&\stackrel{\grad}{\xrightarrow{\hspace*{0.5cm}}}&
\!\!\![C^\infty(\Omega)]^3
\!\!\!&\stackrel{\curl}{\xrightarrow{\hspace*{0.5cm}}}&
\!\!\!  [C^\infty(\Omega)]^3
\!\!\!\!& \stackrel{{\rm div}}{\xrightarrow{\hspace*{0.5cm}}}& 
\!\!\! C^\infty(\Omega)
&
\!\!\!\xrightarrow{\hspace*{0.1cm}}  &0\\
&&&\;\;\;\;\;\;\;\;\;\;\;\;\downarrow {\scriptstyle \pi_{d,0}} & &
\downarrow{\scriptstyle \pi_{d,1}}
& &\downarrow{\scriptstyle  \pi_{d,2}}
& &
\;\;\;\;\downarrow{\scriptstyle \pi_{d,3}}\\
&\mathbb{R}& {\xrightarrow{\hspace*{0.1cm}}}&
M_{d,r}^0(\mct^\s)\!\!\!
&\stackrel{\grad}{\xrightarrow{\hspace*{0.5cm}}}&
{M_{d,r-1}}(\mct^\s) 
\!\!\!&\stackrel{\curl}{\xrightarrow{\hspace*{0.5cm}}}&
\!\!\!M_{d,r-2}^2(\mct^\s)
\!\!\!\!& \stackrel{{\rm div}}{\xrightarrow{\hspace*{0.5cm}}}& 
\!\!\!M^3_{d,r-3}(\mct^\s)
&
\!\!\!\xrightarrow{\hspace*{0.1cm}}  &0.
\end{array}
\end{alignat*}
Similar results hold for the other two  sequences in \eqref{eqn:Globalcomplex3D2} as well.


Note that the top row is an exact sequence if $\Omega$ is contractible; see for example \cite{Costabel}. 
In the next result we will show that the bottom row is also exact on contractible domains. 
Unfortunately,  the projections by themselves do not prove the discrete exactness property because they
require extra smoothness.  However, the exactness of the first and last mapping
can be proved easily, and the the exactness of the second mapping will follow from a counting argument.

\begin{theorem}\label{thm:GlobalComplexABC}
Suppose that $\Omega$ is contractible.  
Then the complexes in \eqref{eqn:Globalcomplex3D2}
are exact.
\end{theorem}
\begin{proof}
We prove exactness of the second sequence \eqref{eqn:Globalcomplex3D22}.  The other two can be proved by similar arguments.

(i) Let $\omega \in M_{d,r-1}^1(\mct^\s)$ with $\curl \omega=0$.   Then using a standard result from \cite{Costabel} (see also \cite{GiraultRaviartBook}) there exists a $\rho \in H^2(\Omega)$ such that $\grad \rho=\omega$. Since $\omega$ is $C^1$ at the vertices,  
$\rho$ is $C^2$ at the vertices. Also on each $T \in \mct$,  $\omega \in M_{d,r-1}^1(T^\s)$, and using that $\grad \rho=\omega$, we have $\rho \in M_{d,r}^0(T^\s)$. Hence, $\rho \in M_{d,r}^0(\mct^\s)$.

(ii) Next, it is shown in \cite{Zhang04} that $\dive:M_{r-2}^2(\mct^\s) \to V_{d, r-3}^3(\mct^\s)$ is a surjection for $r\ge 5$.

(iii) Finally, we show that $\curl:M_{d,r-1}^1(\mct^\s) \to M_{ r-2}^2(\mct^\s)$ is a surjection for $r \ge 5$ using a counting argument. Let $\bbV$, $\bbE$, $\bbF$, and $\bbT$ denote
the number of  vertices,  edges,  faces, 
and tetrahedron in $\mct$, respectively.  
We then set 
\[
{\rm ker}\, M_{r-2}^ 2(\mct^\s):= \{\omega \in M_{r-2}^2(\mct^\s):\ \Div \omega=0\}.
\]
By the rank-nullity theorem, part (i), and Lemmas \ref{lem:C1-DOF} and \ref{lem:H1curlDOFs}, we have that
\begin{align*}
\dim \curl M_{d,r-1}^1(\mct^\s)
&= \dim M_{d,r-1}^1(\mct^\s) - \dim \grad M_{d,r}^0(\mct^\s)\\
&= \dim M_{d,r-1}^1(\mct^\s) - \dim  M_{d,r}^0(\mct^\s)+1\\
&\hspace*{-2cm} = \Big(12 \bbV + \big[3(r-4) + 3(r-3)\big]\bbE+ \big[\frac12 (r-2)(r-3) + (r-3)(r-5)
+(r-3)(r-4)\big]\bbF\\
&\hspace*{-2cm}\qquad + (r-3)(2r-5)(r-4)\bbT\Big)-\Big(10\bbV+\big[(r-5) + 2(r-4)\big]\bbE + \big[\frac12 (r-5)(r-4)\\
&\hspace*{-2cm}\qquad \qquad +\frac12 (r-3)(r-2)\big]\bbF + \frac23 (r-4)(r-3)(r-2)\bbT\Big)+1\\
& = 2\bbV +(3r-8)\bbE + \big(\frac32 r^2-\frac{21}{2}r+17\big)\bbF
+ \big(\frac43 r^3-13 r^2+\frac{125}{3}r-44\big)\bbT+1.
\end{align*}
Likewise, by the rank-nullity theorem, part (ii), and Lemmas  \ref{lem:H1divDOFs}, and \ref{lem:L2DOFs},
we obtain
\begin{align*}
\dim  {\rm ker}\, M_{r-2}^ 2(\mct^\s)
&= \dim M_{r-2}^2(\mct^\s) - \dim V_{d,r-3}^3(\mct^\s)\\
& = \Big(3\bbV +3(r-3)\bbE + \frac32 (r-3)(r-4)\bbF\\
&\qquad + 3 \big[ 1+4(r-3)+3(r-3)(r-4) + \frac46 (r-5)(r-4)(r-3)\big]\bbT\Big)\\
&\qquad - \Big(\frac46 r(r-1)(r-2)\bbT\Big)\\
& = 3\bbV + 3(r-3)\bbE + \frac32 (r-3)(r-4)\bbF
+\big[\frac43 r^3 -13r^2 + \frac{125}{3} r - 45\big]\bbT.
\end{align*}
We then find that
\begin{align*}
\dim {\rm ker}\, M_{r-2}^ 2(\mct^\s) - \dim \curl M_{d,r-1}^1(\mct^\s)
& = \bbV-\bbE+\bbF-\bbT-1=0,
\end{align*}
by an Euler relation.
Since $\curl M^1_{d,r-1}(\mct^\s) \subset {\rm ker}\, M_{r-2}^ 2(\mct^\s)$,
we have $\curl M^1_{d,r-1}(\mct^\s)  = {\rm ker}\, M_{r-2}^ 2(\mct^\s)$,
and therefore the complex \eqref{eqn:Globalcomplex3D2} is exact.
\end{proof}

\section{Concluding Remarks}\label{section-Conclusion}
In this paper we have developed several new local discrete
de Rham complexes with varying level of smoothness on Alfeld
splits.  These results lead, e.g.,  to characterizations of discrete
divergence-free subspaces for the Stokes problem
and local dimension formulas of smooth piecewise polynomial spaces.
We have also constructed analogous global complexes in three dimensions
and projections that commute with the differential operators.
In the future, we plan on using our techniques to study different 
types of splits (e.g., Powell-Sabin, Worsey-Frain) as done in \cite{ChristiansenHu}
for low-order approximations.
In addition, we plan to construct degrees of freedom for the spaces
in any spatial dimension.

\appendix

\section{Proof of Lemma \ref{lem:SmoothAtVerts}}
It is shown in \cite{Alfeld84} that functions in $M_{d,r}^0(T^\s)$ are $C^2$ on $\Delta_0(T)$.
Moreover, it is clear that the result is true in the cases $k=2,3$.  Thus, it suffices to prove the result
$k=1$.  For readability, we prove the result with $r$ replaced by $r-1$.

Define $\tilde{M}_{d,r-1}^1(T^\s) = \{\omega \in M_{d,r-1}^1(T^\s):\ \text{$\omega$ is $C^1$ on $\Delta_0(T^\s)$}\}$.
We show that $\tilde{M}_{d,r-1}^1(T^\s) =M_{d,r-1}^1(T^\s)$.

Let $\kappa\in M_{r-2}^2(T^\s)\subset \Vrc(T^\s)$ satisfy $\Div \kappa=0$.
Using Theorem \ref{thm:AllSequencesC2Exact}
there exists $\omega\in \Mrc(T^\s)$ such that $\kappa = \curl \omega$\footnote{Note that
the proof of Theorem \ref{thm:AllSequencesC2Exact} does not depend on Lemma \ref{lem:SmoothAtVerts}.}.
Because $\kappa$ is continuous, we have $\omega \in M_{d,r-1}^1(T^\s)$,
and so $\omega\in \tilde{M}_{d,r-1}^1(T^\s)$.  Consequently, we easily deduce that
\begin{align*}
\bbR  {\xrightarrow{\hspace*{0.5cm}}}\
{ M}_{d,r}^0(T^\s)\
\stackrel{\grad}{\xrightarrow{\hspace*{0.5cm}}}\
\tilde{M}_{d,r-1}^1(T^\s)\
\stackrel{\curl}{\xrightarrow{\hspace*{0.5cm}}}\
{M}_{r-2}^2(T^\s)\
\stackrel{{\rm div}}{\xrightarrow{\hspace*{0.5cm}}}\
{ V}_{d,r-3}^3(T^\s)\
\xrightarrow{\hspace*{0.5cm}}\
 0
\end{align*}
is exact.  This implies
that 
\[
\dim \tilde{M}_{d,r-1}^1(T^\s) = \dim \grad M_{d,r}^0(T^\s) + \dim M_{r-2}^2(T^\s) - \dim V_{d,r-3}^3(T^\s).
\]
On the other hand, the exactness of \eqref{eqn:AllSequencesC22}
yields
\[
\dim {M}_{d,r-1}^1(T^\s) = \grad M_{d,r}^0(T^\s) + \dim M_{r-2}^2(T^\s) - \dim V_{d,r-3}^3(T^\s),
\]
and therefore $\dim \tilde{M}_{d,r-1}^1(T^\s)  = \dim M_{d,r-1}^1(T^\s)$.
Since $\tilde{M}_{d,r-1}^1(T^\s)  \subset  M_{d,r-1}^1(T^\s)$, we conclude
that $\tilde{M}_{d,r-1}^1(T^\s)  =  M_{d,r-1}^1(T^\s)$. \hfill \qed

\section{Proof Theorem \ref{thm:LocalProjectionOne}}
Property \eqref{commute-1} is the same as \eqref{ABCcommute-1},
so we only need to prove \eqref{commute-2} and \eqref{commute-3}.

(i) {\em Proof of \eqref{commute-2}.}
Let $p\in [C^\infty(\Omega)]^3$, and set
$\omega = \Pi_{d,2} \curl p - \Pi_2 \curl p\in M_{r-2}^2(T^\s)$.
Then using the definitions of $\Pi_{d,2}$ and $\Pi_2$,
and by using the DOFs \eqref{XDOFs1}--\eqref{XDOFs5},
\eqref{eqn:H1DivergenceDOFs1}--\eqref{eqn:H1DivergenceDOFs2},
\eqref{eqn:H1DivergenceDOFs4}, \eqref{eqn:H1DivergenceDOFs6}--\eqref{eqn:H1DivergenceDOFs7}, and \eqref{ABCcommute-2},
we conclude that $\omega$ vanishes on the DOFs of $M_{r-2}^2(T^\s)$, \eqref{XDOFs}.
Thus, applying Lemma \ref{lem:H1divDOFs}, we get $\omega=0$.  Using
\eqref{ABCcommute-2}, we have
\[
\curl \Pi_{d,1} p = \Pi_{d,2} \curl p = \Pi_2 \curl p,
\]
and so \eqref{commute-2} is satisfied.

(ii) {\em Proof of \eqref{commute-3}.}
Let $p\in [C^\infty(T)]^3$ and set $\rho = \Div \Pi_2 p - \Pi_3 \Div p\in V_{d,r-3}^3(T^\s)$.
We show that $\rho$ vanishes on the DOFs \eqref{WDOFs}.
First, by \eqref{WDOFs1}, the divergence theorem, and \eqref{XDOFs3}, we have
\begin{align*}
\int_T
\rho\, \mathrm{dx} = \int_T \big(\Div \Pi_2 p - \Div p\big)\, \mathrm{dx}
= \int_{\p T} \big(\Pi_2 p - p\big)\cdot n\, \mathrm{dA}=0.
\end{align*}
Next, we apply the definitions of $\Pi_2$, $\Pi_3$, and the 
DOFs \eqref{XDOFs5}, \eqref{WDOFs} to obtain
\[
\int_T \rho \kappa\, \mathrm{dx} =0\qquad \forall \kappa\in \mathring{V}_{d,r-3}(T^\s).
\]
Finally, applying Lemma \ref{lem:L2DOFs}, we conclude that $\rho\equiv 0$.
This concludes the proof.\hfill \qed

\section{Proof Theorem \ref{thm:LocalProjectionThree}}
\begin{proof}
(i) {\em Proof of \eqref{DEFcommute-3}.}
Set $\rho = \Div \Pi_{c,2} p - \Pi_3 \Div p\in V_{d,r-3}^3(T^\s)$.  We show that 
$\rho$ vanishes on \eqref{WDOFs}.  We easily find that
\begin{align*}
\int_T \rho \kappa\, \mathrm{dx} = 0\qquad \forall \kappa\in \mathring{V}_{d,r-3}^3(T^\s)
\end{align*}
by \eqref{WDOFs2} and \eqref{eqn:StenbergHDiv5}.  We also have
\begin{align*}
\int_T \rho\, \mathrm{dx} = \int_T \big(\Div \Pi_{c,2} p - \Div p)\, \mathrm{dx} = \int_{\p T} \big(\Pi_{c,2} p - p\big)\cdot n\, \mathrm{dA}=0
\end{align*}
by \eqref{eqn:StenbergHDiv3}.  Thus $\rho$ vanishes on \eqref{WDOFs}, and so $\rho\equiv 0$.

(ii) {\em Proof of \eqref{DEFcommute-2}.} Set $\rho = \curl \Pi_{c,1} p - \Pi_{c,2} \curl p\in \Vrc(T^\s)$.
We show that $\rho$ vanishes on \eqref{eqn:StenbergHDiv}.

We clearly have $\rho(a)=0$ for all $a\in \Delta_0(T)$ by \eqref{eqn:StenbergHDiv1} and  \eqref{eqn:VecHermiteDOFs1}.
Furthermore,
\begin{align*}
\int_e (\rho\cdot n_F)\kappa\, \mathrm{ds}=0\qquad \forall \kappa\in \Delta_{r-4}(e)\quad \forall e\in \Delta_1(F),\ \forall F\in \Delta_2(T),
\end{align*}
by \eqref{eqn:StenbergHDiv2} and \eqref{eqn:VecHermiteDOFs2}.  Next, we apply Stokes Theorem and \eqref{eqn:StenbergHDiv2},
\eqref{eqn:VecHermiteDOFs2}, \eqref{eqn:VecHermiteDOFs6} to get
\begin{align*}
\int_F (\rho\cdot n_F)\kappa\, \mathrm{dA} 
&= \int_F ({\rm curl}_F \Pi_{c,1} p - \curl_F p)\kappa\, \mathrm{dA}\\
& = \int_F (\Pi_{c,1} p - p)\cdot {\rm rot}_F \kappa\, \mathrm{dA} + \int_{\p F} (\Pi_{c,1} p - p)\cdot t \kappa \, \mathrm{ds}\\
& = \int_F (n_F\times (\Pi_{c,1} p - p)\times n_F)\cdot ({\rm grad} \kappa\times n_F)\, \mathrm{dA} + \int_{\p F} (\Pi_{c,1} p - p)\cdot t \kappa \, \mathrm{ds}\\
&=0
\end{align*}
for all $\kappa\in \pol_{r-5}(F)$.  Finally, using \eqref{eqn:StenbergHDiv4} and \eqref{eqn:VecHermiteDOFs8}, we have
\begin{align*}
\int_T \rho \cdot \kappa\, \mathrm{dx} & = 0\qquad \forall \kappa\in \curl \Mrco(T^\s),
\end{align*}
and by \eqref{eqn:StenbergHDiv4}, we have
\begin{align*}
\int_T (\Div \rho)\kappa\, \mathrm{dx}=0\qquad \forall \kappa\in \mathring{V}_{d,r-3}^3(T^\s).
\end{align*}
Thus, $\rho$ vanishes on all the DOFs \eqref{eqn:StenbergHDiv}, and therefore $\rho\equiv 0$.

(iii) {\em Proof of \eqref{DEFcommute-1}.}
Set $\rho = \grad \Pi_{d,0} p - \Pi_{c,1} \grad p\in \Mrc(T^\s)$.
We show that $\rho$ vanishes on \eqref{eqn:VecHermiteDOFs}.

We have $D^\alpha \rho(a)=0$ for all $|\alpha|\le 1$ and $a\in \Delta_0(T)$
by \eqref{eqn:VecHermiteDOFs1} and \eqref{eqn:C1DOF1}, and
\begin{align*}
\int_e \rho\cdot \kappa \, \mathrm{ds} = \int_e \big(\grad \Pi_{d,0} p - \grad p)\cdot \kappa\, \mathrm{ds}=0\qquad \forall \kappa\in [\pol_{r-5}(e)]^3
\end{align*}
by \eqref{eqn:VecHermiteDOFs2} and \eqref{eqn:C1DOF1}--\eqref{eqn:C1DOF22}.  Furthermore, we have
\begin{align*}
\int_e (\curl \rho|_F \cdot n_F)\kappa\, \mathrm{ds} = 0
\end{align*}
by \eqref{eqn:VecHermiteDOFs3}.

Let $\kappa\in \pol_{r-4}(F)$.  Then
\begin{align*}
\int_F (\rho \cdot n_F) \kappa\, \mathrm{dA}
& = \int_F\big( \grad \Pi_{d,0} p - \grad p)\cdot n_F \kappa \, \mathrm{dA}=0
\end{align*}
by \eqref{eqn:VecHermiteDOFs4} and \eqref{eqn:C1DOF32}.  Moreover, we have
\begin{align*}
\int_F (n_F\times \rho \times n_F)\cdot \kappa\, \mathrm{dA}=0\qquad \forall \kappa\in  D_{r-5}(F)
\end{align*}
by using the exact same arguments as those found in the proof of Theorem \ref{thm:LocalProjectionTwo}.

Finally, we apply \eqref{eqn:VecHermiteDOFs7} and \eqref{eqn:C1DOF4} to get
\begin{align*}
\int_T \rho \cdot \kappa \, \mathrm{dx}=0\qquad \forall \kappa\in \grad \mathring{M}^0_{d,r}(T^\s),
\end{align*}
and use \eqref{eqn:VecHermiteDOFs8} to get
\begin{align*}
\int_T \curl \rho \cdot \kappa \, \mathrm{dx}=0\qquad \forall \kappa\in \curl \Mrco(T^\s).
\end{align*}
Thus, $\rho$ vanishes on the DOFs \eqref{eqn:VecHermiteDOFs}, and thus $\rho\equiv 0$.
\end{proof}

\end{document}